\documentclass[EJP,preprint]{ejpecp} % replace ECP by EJP if needed.
\usepackage{bbm}
\usepackage{amsmath,amsthm,amsfonts,amssymb}
\usepackage[capitalise,noabbrev]{cleveref}

\usepackage{subcaption}

\newcommand{\rel}[1]{\stackrel{#1}{\sim}}
\renewcommand{\P}[0]{\mathbb{P}}
\newcommand{\E}[0]{\mathbb{E}}
\newcommand{\cP}[2]{\mathbb{P}\left(#1\ \middle|\ #2\right)}
\newcommand{\cE}[2]{\mathbb{E}\left[#1\ \middle|\ #2\right]}
\newcommand{\nrel}[1]{\stackrel{#1}{\not\sim}}
\newcommand{\bigO}{\mathcal{O}}

\usepackage{macros}

\newcommand{\PPM}{\mathrm{PPM}}

\newcommand{\Binom}{\mathrm{Bin}}
\newcommand{\agreement}{\mathrm{A}}
\newcommand{\agreementNormalized}{\tilde{\mathrm{A}}}
\newcommand{\balanced}{\mathrm{Balanced}}
\providecommand{\citep}[1]{\cite{#1}}
\providecommand{\citet}[1]{\cite{#1}}

\SHORTTITLE{Recovering Small Communities}

\TITLE{Recovering Small Communities in the Planted Partition Model} % \thanks is optional. Insert line breaks with \\

\AUTHORS{%
  Martijn~Gösgens\footnote{University of Twente, Enschede, The Netherlands.
    \EMAIL{research@martijngosgens.nl}}\orcid{0000-0002-7197-7682}
  \and %% remove this line and below if single author
  Maximilien~Dreveton\footnote{Université Gustave Eiffel, Paris, France. \BEMAIL{maximilien.dreveton@univ-eiffel.fr}}\orcid{0000-0001-6613-0615}}%AUTHORS
 % Separate items with ;

\AMSSUBJ{05C80; 05C85} % Edit. Separate items with ;
\SUBMITTED{March 2, 2026} % Edit.
\ACCEPTED{??? ??, 202?} % Edit.

\VOLUME{0}
\YEAR{2023}
\PAPERNUM{0}
\DOI{10.1214/YY-TN}

\ABSTRACT{ %We analyze community recovery in the planted partition model (PPM) in regimes where the number of communities is arbitrarily large. We examine the three standard recovery regimes: exact recovery, almost exact recovery, and weak recovery. When communities vary in size, traditional accuracy- or alignment-based metrics become unsuitable for assessing the correctness of a predicted partition. To address this, we redefine these recovery regimes using the correlation coefficient, a more versatile metric for comparing partitions. We then demonstrate that \emph{Diamond Percolation}, an algorithm based on common-neighbors, successfully recovers communities under mild assumptions on edge probabilities, with minimal restrictions on the number and sizes of communities. As a key application, we consider the case where community sizes follow a power-law distribution, a characteristic frequently found in real-world networks. To the best of our knowledge, we provide the first recovery results for such unbalanced partitions.

We study community recovery in the planted partition model in regimes where the number and sizes of communities may vary arbitrarily with the number of vertices. In such highly unbalanced settings, standard accuracy or overlap-based metrics become inadequate for assessing recovery performance. Instead, we propose the \emph{correlation coefficient} between partitions as a recovery metric, which remains meaningful even when the number or sizes of communities differ substantially. We then analyze a simple common-neighbor–based clustering rule which groups two adjacent vertices if they share more than one common neighbor. We establish explicit recovery conditions under sparse inter-community connectivity, without requiring prior knowledge of the model parameters. In particular, in graphs of size~$n$, this algorithm achieves \emph{exact recovery} for communities with sizes $\Omega(\log n)$, \emph{almost exact recovery} for sizes $\omega(1)$ and \emph{weak recovery} for sizes $\Omega(1)$. In contrast to most existing results, which assume (nearly) balanced communities, our method successfully recovers \emph{small} and \emph{heterogeneously-sized} communities, and improves existing guarantees even in some balanced settings. Finally, our results apply to community sizes that follow a \emph{power-law} distribution, a characteristic frequently found in real-world networks.}

\begin{document}
\section{Introduction}

Community detection in random graphs is a central topic in random graph theory, with applications ranging from social and biological networks to information and communication systems. The planted partition model (PPM) provides a canonical framework for studying this problem. In its simplest form, the vertex set $[n]$ is divided into latent communities, and edges between vertex pairs appear independently with probability $p_n$ if both vertices belong to the same community and with probability $q_n$ otherwise. The statistical and computational thresholds for recovering the planted partition are now well understood when the number of communities is finite or grows slowly, and when communities are approximately balanced in size (see~\citet{abbe2018community} for a survey). 

Despite the large literature on community recovery in the PPM, nearly all existing works make at least one of the following two assumptions: (i) the number of communities is finite or grows slowly with the number of vertices, and/or (ii) the community sizes are asymptotically of the same order. However, these assumptions are often unrealistic in practice and restrictive from a probabilistic viewpoint. For instance, if the partition of the vertex set is chosen uniformly at random from the set of all partitions of $[n]$, the number of communities grows as $n/\log n$ \citep{pittel1997random,sachkov1997probabilistic}. Moreover, empirical studies show that real-world networks contain many small groups and a few large ones, with community sizes following heavy-tailed, often power-law, distributions~\citep{lancichinetti2008benchmark,stegehuis2016power}. From a modeling perspective, it is therefore natural to consider the PPM in regimes where the number of communities can grow arbitrarily with $n$, and where community sizes may differ by orders of magnitude. Such regimes challenge both the usual probabilistic tools and the classical metrics for assessing recovery, since standard accuracy or overlap measures depend implicitly on the number and relative sizes of communities. 

This work studies recovery in the planted partition model under minimal structural assumptions on the latent partition. We allow the number and sizes of communities to be arbitrary, including highly unbalanced cases where small and large groups coexist. To quantify recovery, we employ the correlation coefficient between partitions, a symmetric similarity index with a well-defined random baseline that is studied in~\citet{gosgens2021systematic}. In contrast to accuracy or overlap-based measures, the correlation remains interpretable even when the number of communities in the estimated partition differs from that in the ground truth, and it admits a clean probabilistic formulation amenable to theoretic analysis.

We then analyze a particularly simple community detection procedure. Given a graph $G$, the algorithm constructs a filtered graph $G^*$ that retains only the edges that participate in at least two triangles, and outputs the connected components of $G^*$ as the estimated communities. This rule can be viewed as a local common-neighbor thresholding scheme that does not require knowledge of the model parameters $p_n$ or $q_n$ nor the number of communities. 
%While the simplicity of the algorithm entails certain sparsity restrictions on inter-community edges, it enables a precise probabilistic analysis even in the presence of an arbitrary number of communities.
%Although the simplicity of the algorithm entails certain sparsity restrictions on inter-community edges, it enables a precise probabilistic analysis even in the presence of an arbitrary number of communities. %a clean and general probabilistic analysis over very general partitions.
%\textcolor{red}{I don't understand this last sentence.}

Our main results establish conditions under which the algorithm achieves exact, almost exact, and weak recovery in the PPM. Exact recovery means that the algorithm perfectly infers the planted partition, grouping all vertices correctly with high probability. Almost exact recovery allows for a vanishingly small error. Weak recovery ensures that the inferred partition is still meaningfully correlated with the true communities, performing strictly better than a random guess. 

 We focus primarily on establishing conditions under which Diamond Percolation succeeds at community recovery while imposing minimal structural assumptions on both the number and the sizes of communities. This objective contrasts with a complementary line of work that seeks to characterize information-theoretic recovery thresholds, \textit{i.e.,} regimes in which no algorithm can succeed, but typically under stronger assumptions on the community structure.

In particular, while the simplicity of the algorithm allows us to obtain recovery guarantees with minimal assumptions on the partition, it nonetheless requires some conditions on the sparsity parameters. In particular, we cannot expect to recover the communities in standard regimes, where $p_n$ and $q_n$ scale as $\Theta(1/n)$ or as $\Theta(\log n / n)$, but the number of communities is constant. Because these regimes have been extensively studied, they are not the focus on our work. Moreover, when there is an unbounded number of communities of same-size, exact recovery is impossible when $p_n \asymp q_n = \Theta( \log n / n)$~\cite[Corollary~3]{chen2016statistical}. Likewise, if the communities have constant sizes, scaling $p_n=o(1)$ implies that almost all communities have zero inter-community edges. This motivates moving away from the standard setting by scaling $p_n$ and $q_n$ so that communities are sufficiently dense to allow recovery. Finally, we also note that for several inference problems where the number of communities is unbounded, triangle-counting algorithms have been shown to be near-optimal (see~\cite{chen2016statistical,rush2022easier} and Section~\ref{section:discussion} for further discussion).

A key technical ingredient is a probabilistic refinement lemma showing that, under suitable sparsity conditions on $p_n$ and $q_n$, the partition produced by the algorithm is with high probability a refinement of the true partition. We further prove that, whenever one partition refines another, the correlation coefficient between them concentrates around the square root of the ratio of their intra-community pair counts. This connection enables the correlation metric to be analyzed through simple moment conditions on community sizes and may be of independent interest for future research.

As a key application, we specialize our recovery results to power-law community size distributions, a setting that has not received any theoretical attention despite its empirical relevance. We prove that under mild growth conditions on the number of communities and appropriate scaling of $p_n$, the algorithm achieves exact, almost exact, or weak recovery depending on the sparsity regime. These are the first rigorous recovery guarantees for the PPM with power-law community sizes.

\medskip
\textbf{Notation.} Throughout this paper, $G_n$ denotes a graph with vertex set $[n]=\{1,\dots,n\}$. For $i,j\in[n]$, we write $i\rel{G_n}j$ if $i$ and $j$ are connected by an edge in $G_n$. $T_n$ denotes a partition of $[n]$ that represents the communities of $G_n$. If vertices $i,j\in[n]$ are part of the same community in $T_n$, we denote this by $i\rel{T_n}j$. To avoid cluttering notation, we occasionally omit the subscript $n$ and write $i\rel{T}j$ or $i\rel{G}j$ instead. The set of vertices that are in the same community as $i\in[n]$ is denoted by $T_n(i)$. We denote a vertex chosen uniformly at random from $[n]$ by $I_n$ and denote the size of its community by $S_n=|T_n(I_n)|$.
We denote the number of \emph{intra-community pairs} by
\[
m_{T_n}=\#\left\{1\le i<j\le n\ :\ i\rel{T_n}j\right\}.
\]
We denote the partition of \emph{detected} communities by $C_n$, and define $\rel{C_n}$ and $m_{C_n}$ similarly as above.
We use standard asymptotic notation. For two sequences $a_n,b_n$, we write $a_n\ll b_n$ if $a_n/b_n\to0$; $a_n\gg b_n$ if $b_n/a_n\to0$; and $a_n\sim b_n$ if $a_n/b_n\to1$. We additionally use standard Landau notation: we write $o(a_n)$ to denote any sequence $b_n\ll a_n$; $\omega(a_n)$ to denote any sequence $b_n\gg a_n$; $b_n=\bigO(a_n)$ if $\lim\sup_{n\to\infty}|b_n/a_n|<\infty$; $b_n=\Omega(a_n)$ if $\lim\inf_{n\to\infty}|b_n/a_n|>0$; and $b_n=\Theta(a_n)$ if $a_n=\bigO(b_n)$ and $b_n=\bigO(a_n)$.

We say that an event $A_n$ occurs \emph{with high probability} (or \emph{w.h.p.} in short) if $\P(A_n)\to1$. We say that a sequence of random variables $X_n$ converges \emph{in probability} to $X$ (denoted $X_n\wprto X$) if for any $\varepsilon>0$, $|X_n-X|<\varepsilon$ holds with high probability. We say that a sequence of random variables $X_n$ converges \emph{in distribution} to $X$ (denoted $X_n\stackrel d{\to}X$) if $\mathbb{P}(X_n\le x)\to\P(X\le x)$ for any $x$ where $\P(X\le x)$ is continuous.

\medskip
\textbf{Structure of the paper.} 
The paper is structured as follows. We introduce the formal setting and correlation-based recovery criteria in \cref{section:motivation}. The main assumptions and the algorithm are described in \cref{section:framework_algorithm}. We present the main recovery results for general partitions in \cref{section:arbitrary_partitions}, and apply them to power-law partitions in \cref{section:power_law_partitions}. 
Numerical experiments are presented in \cref{section:experiments}. 
Finally, we conclude the paper with a discussion in \cref{section:discussion}. The proofs can be found in the appendix.

\section{Problem Setting}
\label{section:motivation}

 In this section, we formalize the planted partition model and introduce the recovery criteria used throughout the paper. Our goal is to understand how accurately one can recover a latent community structure when the number and sizes of communities may vary arbitrarily with the total number of vertices. 

 Let $G_n = ([n], E_n)$ be an undirected random graph generated according to a planted partition model (PPM). The vertex set $[n]$ is partitioned into communities described by a random partition $T_n$. Conditioned on $T_n$, edges between vertex pairs $(i,j)$ are drawn independently:
 \begin{align*}
    \cP{i\rel{G_n} j}{T_n}
    \weq
    \begin{cases}
     p_n & \text{ if } i \rel{T_n} j, \\
     q_n & \text{ otherwise,}
    \end{cases}
 \end{align*}
 where $i \rel{G_n} j$ indicate that an edge is present between $i$ and $j$, and $i \rel{T_n} j$ indicate that $i$ and $j$ belong to the same community.  The goal is to infer the latent partition $T_n$ from the observed graph. 
 In \cref{section:recovery_criteria}, we discuss the recovery criteria that we consider in this work. In \cref{subsection:agreement-vs-correlation}, we motivate our choice for the correlation coefficient as performance measure. We discuss random partitions in~\cref{subsection:arbitrary_partitions}. Finally, we discuss relevant PPM literature in~\cref{subsection:special_cases_PPM}.

\subsection{Recovery Criteria}
\label{section:recovery_criteria}

 When the number or sizes of communities are heterogeneous, standard metrics such as agreement or normalized overlap become difficult to interpret, as they depend on the labeling and the number of communities. To overcome this limitation, we measure recovery quality using the correlation coefficient between partitions. This metric has been studied in~\citet{gosgens2021systematic}, where it was shown that it has several desirable properties that other metrics do not have. One of these desirable properties is that it has a \emph{constant baseline}: if a partition $C_n$ is drawn from a distribution that is symmetric w.r.t. vertex permutations, then the expected similarity between $C_n$ and $T_n$ is $0$.
 Another reason for using this correlation coefficient is its analytical tractability under unbalanced partitions (see also \cref{subsection:agreement-vs-correlation}). 

Let $N={n\choose 2}$ denote the number of vertex pairs. Given a partition $C$ of~$[n]$, we let 
\[
 m_C \weq \#\left\{ ij\ :\ i\rel{C}j \right\}
\]
be the number of intra-community pairs of~$C$. The quantity $m_C$ is nonnegative and upper-bounded by $N$, where $m_C=0$ corresponds to $n$ communities of size~$1$, and $m_C=N$ corresponds to a single community of size $n$. Finally, for two partitions $C$ and $T$, we define
\[
m_{CT} \weq \#\left\{ ij\ :\ i\rel{C}j  \text{ and } i\rel{T}j \right\}.
\]
Note that $m_{CT}$ is upper-bounded by the minimum of $m_C$ and $m_T$, where $m_{CT}=m_C=m_T$ occurs if and only if $C=T$.

The \textit{correlation} $\rho(C,T)$ between the two partitions $C$ and $T$ is defined as the Pearson correlation between the indicators $\mathbbm{1}\{i\rel{C}j\}$ and $\mathbbm{1}\{i\rel{T}j\}$ for a vertex-pair $ij$ chosen uniformly at random~\citep{gosgens2021systematic}. It is given by
\begin{equation}
 \label{eq:correlation}
 \rho(C,T) \weq \frac{m_{CT}N-m_Cm_T}{\sqrt{m_C\cdot(N-m_C)\cdot m_T\cdot(N-m_T)}}.
\end{equation}
This correlation lies in the interval $[-1,1]$. The case $\rho(C,T)=1$ occurs iff $C=T$. Conversely, $\rho(C,T)=-1$ implies that $C$ and $T$ are maximally dissimilar, \textit{i.e.}, $i\rel{C}j\Leftrightarrow i\nrel{T}j$ for all $i,j\in[n]$. This can only occur when one of the two partitions corresponds to a single community of size $n$, while the other corresponds to $n$ singleton communities. The correlation coefficient has the convenient property that if $C$ is \emph{uncorrelated} to $T$, then $\rho(C,T)\approx 0$. More precisely, if $T$ is fixed with $0<m_T<N$ and $C$ is sampled from a distribution that is symmetric w.r.t. vertex permutations, then $\E[\rho(C,T)]=0$. 
In addition, if $T_n$ is a sequence of non-trivial partitions\footnote{By non-trivial partition, we exclude two border cases: the partition composed only of singletons and the partition with a single community. This ensures that $0<m_{T_n}<{n\choose 2}$.} and $C_n$ is a sequence of random partitions, each sampled from a vertex-symmetric distribution, then $\rho(C_n,T_n)\wprto0$. This is known as the \emph{constant baseline property}~\citep{gosgens2021systematic}. 

Denote the true partition by $T_n$, with which $G_n$ is sampled, and denote the estimated partition by $C_n = \cD(G_n)$. We say that $\cD$ achieves:
\begin{itemize}
    \item exact recovery if $\P_{G_n,T_n}(\rho(C_n,T_n)=1)\rightarrow1$;
    \item almost exact recovery if $\rho (C_n,T_n) \wprto 1$;
    \item weak recovery if $\rho(C_n,T_n)\ge\rho_0+o_\P(1)$ for some $\rho_0 > 0$. That is, for every $\varepsilon>0$,
    \[
    \P(\rho(C_n,T_n)<\rho_0-\varepsilon)\to0.
    \]
\end{itemize}
 Our recovery criteria differ slightly from the definitions commonly used in the literature, as we use the correlation coefficient instead of the agreement (also known as accuracy). We discuss this choice in Section~\ref{subsection:agreement-vs-correlation}.

\subsection{Agreement versus Correlation}\label{subsection:agreement-vs-correlation}
In this section, we motivate the recovery criteria based on the correlation coefficient and explain why they are more suitable in our setting than other criteria used in the literature. The recovery conditions are commonly defined using agreement rather than correlation (see~\cite[Section~2.3]{abbe2018community} for example).
Consider two partitions $T$ and $C$, each with the \textit{same} number~$k$ of communities. Two vectors $z, z' \in [k]^n$ can represent these partitions. We then define the agreement (also called accuracy) and the normalized agreement (also called overlap) between~$C$ and~$T$ as follows: 
\begin{align*}
 \agreement( C, T ) & \weq \max_{\pi \in \Sym(k)} \frac1n \sum_{i=1}^n \1( z_i = \pi(z'_i)), \\ 
 \agreementNormalized( C, T ) & \weq \max_{\pi \in \Sym(k)} \frac1k \sum_{a=1}^k \frac{ \sum_{i \in [n] } \1(z_i = a, \pi(z'_i) = a )  }{ \sum_{i \in [n] } \1(z_i = a ) },
\end{align*}
where $\Sym(k)$ is the set of permutations of $[k]$. 
A major disadvantage of (normalized) agreement is that it is only defined when $C$ and $T$ consist of the same number of communities. In practice, one typically does not know the exact number of communities, so that one cannot guarantee $C$ to have the same number of communities as $T$.
The correlation coefficient does not suffer from this defect because it is based on the representation of $T$ as a binary relation, instead of the labeling-based representation that agreement is based on.
This allows us to meaningfully measure the similarity between $C$ and $T$ even when their number of communities differ significantly.

Additionally, the correlation coefficient is one of the most effective metrics for comparing partitions. In particular, the correlation coefficient has the \emph{constant baseline} property~\cite{gosgens2021systematic}, which ensures that $\E[\rho(C,T)]=0$ whenever $C$ is uncorrelated to $T$. In contrast, if $C$ and $T$ are uncorrelated and each have $k$ communities, then $\E[\agreement(C,T)]\ge\tfrac1k$, with the exact value of the expectation depending on the sizes of the communities.
The definition of weak recovery is linked to the idea of outperforming random guessing. Therefore, to see if an agreement value is better than a random guess, we should compare it to this size-dependent expected value. In contrast, for the correlation measure, we only need to check if it is strictly positive.

Moreover, unlike agreement-based metrics, the correlation coefficient avoids the need to minimize over permutations of the community labels, which makes some arguments in the proofs tedious. Finally, the correlation coefficient is simple enough to facilitate rigorous theoretical analysis, making it well-suited for both practical and theoretical studies.

The Adjusted Mutual Information (AMI) is another widely used metric for comparing partitions that has many desirable properties~\citep{gosgens2021systematic,vinh2009information}. However, theoretically analyzing the AMI of a partition produced by an algorithm relative to the true partition is highly challenging. Indeed, the AMI is an `adjusted-for-chance' metric, and this adjustment introduces a term that complicates the theoretic analysis~\citep{vinh2009information}. Additionally, computing the AMI has a time complexity of $\bigO(n\cdot k)$, where $k$ is the number of communities~\citep{romano2016adjusting}. In cases where $k = \bigO(n)$—such as those considered in~\cref{thm:weak}—the time complexity becomes $\bigO(n^2)$, which may even be higher than the complexity of Algorithm~\ref{alg:commonNeighborsPartitioning}. In contrast, the correlation coefficient has time complexity $\bigO(n)$.
For these reasons, we formulate our recovery criteria in terms of the correlation coefficient rather than AMI.

\subsection{Random Partitions}
\label{subsection:arbitrary_partitions}

The main contribution of this work is to establish exact, almost exact, and weak recovery conditions in the planted partition model where the latent partition has an arbitrary number of communities with arbitrary sizes. In this section, we highlight some examples of random partitions.

A \textit{single community partition} consists of a single community of size $s_n$, formed by selecting $s_n$ vertices uniformly at random, while each of the remaining $n-s_n$ vertices form singleton communities.

In a \textit{balanced partition}, the vertices are divided into $k$ communities of equal size $s$, for $k\cdot s\le n$.
We place the remaining $n-k\cdot s$ vertices into singleton communities, so that the partition consists of $n-k(s-1)$ communities.
These partitions are denoted by $T_n\sim\balanced(n,k,s)$.
For $k=1$, this corresponds to a single community partition.

In the \textit{uniform partition}, $T_n$ is chosen uniformly from all partitions of $[n]$. This distribution has been extensively studied, and many of its asymptotic properties are known~\citep{harper1967stirling,pittel1997random}. For example, it is known that the number of communities grows as $n/\log n$. We denote this distribution by $T_n\sim\text{Uniform}(n)$.

A \textit{multinomial partition} is constructed by assigning each vertex $i\in[n]$ independently to a community $a\in[k_n]$ with probability $\pi_a$, where $(\pi_a)_{a\in[k_n]}$ is a given probability sequence.
By specifying different probability sequences, this allows one to construct a broad range of partition distributions. %In particular, in this paper, we show how to construct a multinomial partition so that the community size distribution follows a \emph{power-law}. 

%\paragraph{Power-law partitions}
As mentioned in the introduction, community sizes typically follow a power-law distribution. Such partitions can be sampled from a multinomial partition as follows. Let $\tau>2$ and consider a sequence of i.i.d. $\text{Exp}(1)$ random variables $(X_a)_{a\in[k_n]}$ and take the multinomial partition corresponding to the random probability sequence given by
 \[
  \Pi_a \weq \frac{e^{X_a/\tau}}{\sum_{b\in[k_n]}e^{X_b/\tau}}.
 \]
 We show in Theorem~\ref{thm:power-law-rescaled} that the sizes of the communities obtained from such random partition follow a power-law distribution. %We denote a partition $T_n$ drawn from this distribution by $T_n\sim\text{Powerlaw}(\tau,k_n,n)$.

\subsection{Related Works and Special Cases of the PPM}
\label{subsection:special_cases_PPM}

%\paragraph{Planted partition model and its variants}

%Our primary model is the \textit{planted partition model}, defined by a partition $T_n$ of $[n]$ into any number of communities of varying sizes. 
The planted partition model encompasses several well-known special cases. When the partition~$T_n$ consists of a single community of size $s_n \ge 2$ and all other communities are singletons (size 1), we recover the \textit{planted dense subgraph} model. The \textit{planted clustering model} arises when $T_n$ contains $k$ communities of equal size $s_n \ge 2$ while the remaining $n-s_n \cdot k_n$ vertices are singletons. When all community sizes are greater or equal to 2, we recover the \textit{stochastic block model} (SBM) with homogeneous interactions. %In this section, we review some of the existing recovery results for these models. 

When the vertex set is partitioned into $k = \Theta(1)$ communities each of size $\Theta(n)$, tight conditions both for exact and weak recovery are available in the literature, and we refer to~\citet{abbe2018community} for a review. In the following, we only highlight the results when $k$ can grow with $n$. 

\medskip
\textbf{Exact recovery.} \citet{chen2016statistical} establish several key results for the impossibility and possibility of exact recovery when the communities are of equal size $n/k$ and~$k$ grows arbitrarily. Their paper highlights various phase transitions and--up to unspecified constants--precisely characterize those transitions. The problem progresses through four distinct stages: (1) being statistically unsolvable, (2) becoming statistically solvable but computationally expensive, (3) transitioning to being solvable in polynomial time, and finally (4) being solvable by a simple common-neighbor counting algorithm. 
%The first phase transition concerns the impossibility of exact recovery by any algorithm, the second by polynomial-time algorithms, and the third by simple algorithms only based on degree counting. 
However, the equal-size community assumption is limiting, as we highlighted in the introduction: communities in real networks can have sizes with different orders of magnitude.
Moreover, the algorithms in~\citet{chen2016statistical} require knowledge of the number of communities and of $p_n, q_n$. In contrast, we establish that Algorithm~\ref{alg:commonNeighborsPartitioning}, a simple common-neighbor counting algorithm, can achieve exact recovery even when the communities have arbitrary sizes and the number of communities is unknown. 

\medskip
\textbf{Weak recovery.}
 Weak recovery for multinomial partitions with a uniform prior $\pi_a = 1/k_n$ has been investigated in recent preprints \citep{carpentier2025phase,chinstochastic,luo2023computational}. In particular, using non-backtracking walk statistics, \citet{chinstochastic} shows that weak recovery is achievable below the Kesten--Stigum threshold in the regime $\sqrt{n} \ll k_n \ll n$, $p_n = \Theta(k_n / n)$ and $q_n = \Theta(1/n)$, and they propose a conjecture for the weak recovery threshold in this regime. The authors also propose a conjecture characterizing the precise weak recovery threshold in this setting. Supporting evidence for this conjecture is provided in \citet{carpentier2025phase}, which analyzes the problem within the low-degree polynomial framework and further extends the weak recovery region to denser regimes by incorporating clique-based statistics.

Moreover, \citet{luo2023computational} establishes low-degree hardness results for weak recovery in SBMs with multinomial partitions under a uniform prior $\pi_a = 1/k$. Specifically, \cite[Theorem~5]{luo2023computational} shows that when $k \le \sqrt{n}$ and $\frac{n(p_n-q_n)^2}{k^2 q_n(1-p_n)} = o(1)$, no low-degree polynomial algorithm can achieve weak recovery. In the sparse regime $q_n = \Theta(n^{-1})$, this condition simplifies to $p_n \ll k/n$, in which case a typical vertex has no neighbors within its own community. When $k \ge \sqrt{n}$, \cite[Theorem~5]{luo2023computational} instead yields the impossibility condition $p_n^2 \ll q_n$, which further highlights that $p_n$ and $q_n$ cannot be asymptotically of the same order if weak recovery is to be achievable. 

\medskip
\textbf{Planted dense subgraph} \citet{chen2016statistical} establishes several key results for the impossibility and
possibility of exact recovery when the partition comprises a single community of size $s_n \ge 2$ and all other communities are singletons. However, they establish the possibility of exact recovery under the additional assumption $s \ge \log n$ while \cref{thm:exact} can be applied for $s_n \gg 1$. 

\citet{Schramm_Wein_2022} establishes criteria for the success and failure of low-degree polynomials in achieving weak recovery. In particular, polynomials of degree~$n^{\Omega(1)}$ fail at weak recovery if $\frac{p_n-q_n}{ \sqrt{q_n(1-p_n)}} \ll \min \{ 1, \frac{\sqrt{n}}{s_n} \}$. Conversely, polynomials of degree~$\bigO(\log n)$ succeed at weak recovery if $\frac{p_n-q_n}{\sqrt{q_n}} \gg \frac{\sqrt{n}}{s_n}$ and $p_n s_n =\omega(1)$. In the regime $p_n = \Theta(n^{-a})$, $q_n = \Theta(n^{-a})$ and $s_n=\Theta(n^{b})$ for constants $a\in(0,2)$ and $b \in (0,1)$, this implies low-degree hardness of recovery at degree $n^{\Omega(1)}$ whenever $b<(1 + a)/2$, while low-degree polynomials succeed whenever $b>(1 + a)/2$. For related results on weak recovery in the planted dense subgraph model and its connection to the planted clique problem, we refer to~\citet{hajek2015computational}. 
%\textcolor{red}{Do you mean $s_n=\Theta(n^{b})$ (no minus in exponent)? Then I get low-degree hardness for $1/2<b<(1-a)/2$, which is an empty interval. Maybe the clue here is that $p_n\gg q_n$ is necessary for $s_n\ll n$? More precisely, we need $p_n=\Omega( \sqrt{q_n}\min \{ 1, \frac{\sqrt{n}}{s_n} \})$, which is the same result as from \cite{luo2023computational}} \MDnote{Yes I think my mistake is that in their paper $n \rho $ is our $s_n$.}
%Finally, we observe that when $p = \bigO(n^{-1})$, the condition $p s /n \gg 1/n$, necessary for the success of low-degree polynomials, becomes $s \gg n$ and is thus never verified.

\medskip
\textbf{Estimation of the number of communities}
Several methods have been proposed to address the model selection problem, specifically the estimation of the number of communities $k$. However, most theoretical works assume that the number of communities is either finite or, at best, grows slowly (typically, $k = o(\log n)$). 
The only work that imposes no restrictions on the number of communities appears to be \citet{cerqueira2020estimation}. In this paper, the author proves that an information-theoretic criterion, based on the Minimum Description Length principle, is strongly consistent. However, computing this estimator is NP-hard. In contrast, in the regimes where Algorithm~\ref{alg:commonNeighborsPartitioning} achieves exact recovery, it provides an efficient method for recovering both the communities and their number, without prior knowledge of $k$. Moreover, when the assumptions of \cref{thm:refinement-generalized} are satisfied, the partition inferred by Algorithm~\ref{alg:commonNeighborsPartitioning} is composed of more communities than the true partition, which provides an upper-bound on $k$.

\section{Theoretical Framework and Algorithm}
\label{section:framework_algorithm}
In this section, we present Diamond Percolation and discuss some of its properties. In addition, we formulate the assumptions that we make in order to prove the recovery criteria in Sections~\ref{section:arbitrary_partitions} and~\ref{section:power_law_partitions}. 

\subsection{Diamond Percolation}
\label{subsection:algorithm}

 Consider an unweighted and undirected graph $G$ with vertex set $[n]$, and let $W_{ij}$ denote the number of common neighbors between $i$ and $j$ (i.e., the number of \emph{wedges} from $i$ to $j$). That is, 
 \[
  W_{ij} \weq \#\left\{u\in[n]\setminus\{i,j\}\ :\ u\rel{G} i \text{ and } u \rel{G} j\right\}.
 \]

 We consider Algorithm~\ref{alg:commonNeighborsPartitioning} for detecting communities. This algorithm first constructs a graph $G^*$ such that $i\rel{G^*}j$ iff $i\rel{G}j$ and $W_{ij}\ge2$. In other words, $G^*$ only keeps the edges of $G$ that correspond to the middle edge of some \emph{diamond}, i.e., if it is part of at least two triangles. We then consider the partition $C$ formed by the connected components of $G^*$ and return these as the detected communities. In the rest of the paper, we denote Algorithm~\ref{alg:commonNeighborsPartitioning} by $\cD$ and denote the resulting partition into communities by $C = \cD(G)$. Note that the algorithm $\cD(\cdot)$ does not require knowledge of any model parameters. Algorithm~\ref{alg:commonNeighborsPartitioning} is illustrated in \cref{fig:example}.
\begin{figure}[!ht]
    \centering
    \includegraphics[width=0.5\linewidth]{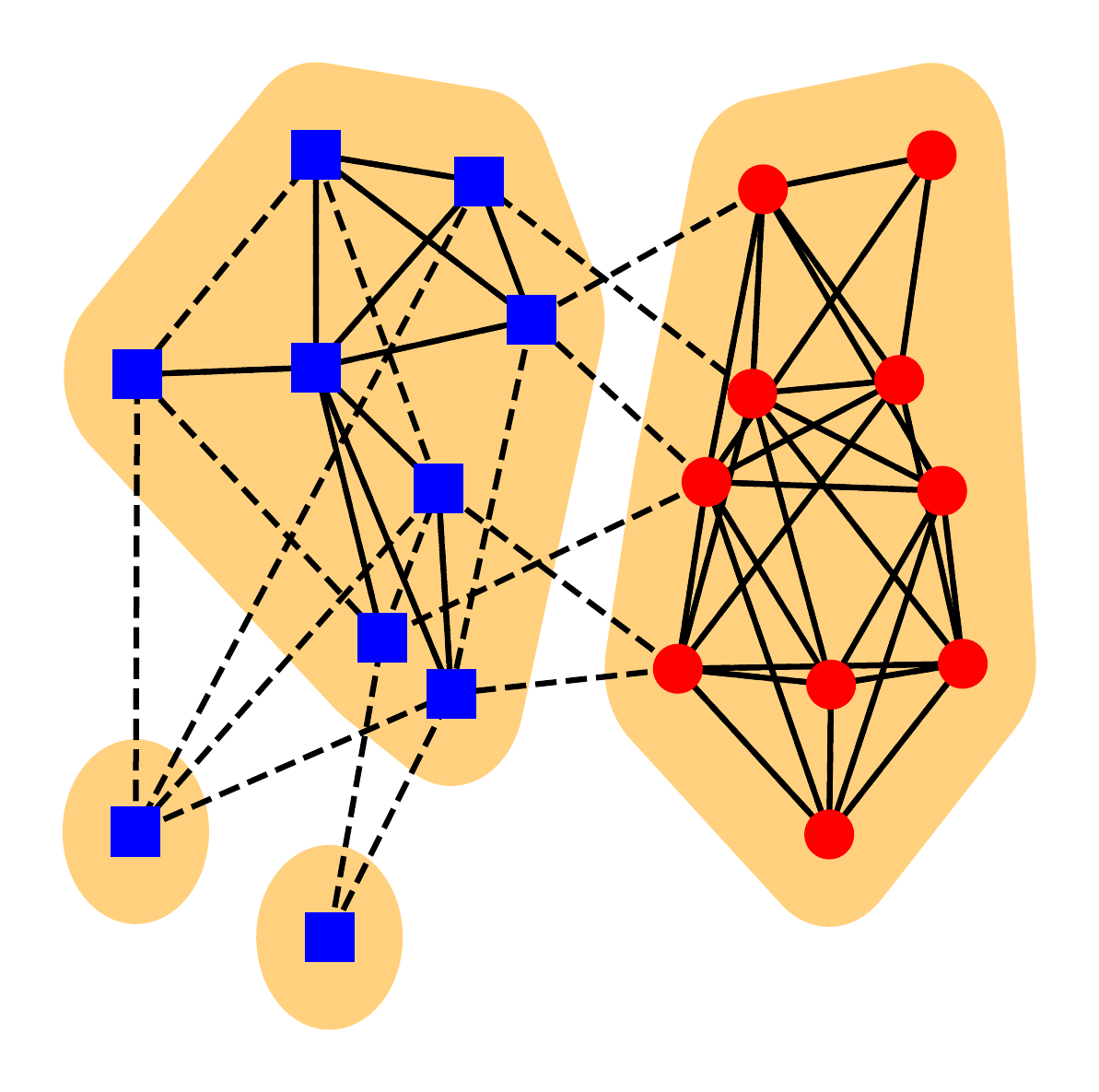}
    \caption{Algorithm~\ref{alg:commonNeighborsPartitioning} is illustrated on a PPM consisting of two equally-sized communities of size $10$ each, with $p=\tfrac12$ and $q=\tfrac1{20}$. The true communities correspond to the red circles and blue squares. The solid lines are the edges of $G^*$, while the dashed lines are the edges of $G$ that are not retained in $G^*$. The orange shaded regions represent the detected communities. We see that the two communities are correctly separated, but that two vertices are incorrectly isolated.}
    \label{fig:example}
\end{figure}

\begin{algorithm}[!ht]
\small
\DontPrintSemicolon
\KwIn{Graph $G = ([n], E)$}
\BlankLine
Let $E^* = \emptyset$
\For{ $ij \in E$}{
Let $W_{ij} =\#\left\{u\in[n]\setminus\{i,j\}\ :\ u\rel{G} i \text{ and } u \rel{G} j\right\} $ be the number of common neighbors between $i$ and $j$. 

\If{ $W_{ij} \ge 2$}{
$E^* = E^* \cup \{ ij \}$
}
}
\BlankLine
\KwOut{Partition formed by the connected component(s) of~$G^* = ([n], E^*)$.} 
\caption{Diamond Percolation.}
\label{alg:commonNeighborsPartitioning}
\end{algorithm}

The following lemma provides the space and time complexity of Algorithm~\ref{alg:commonNeighborsPartitioning}. 
\begin{lemma}\label{lemma:running_time_algo}
    Algorithm~\ref{alg:commonNeighborsPartitioning} has $\bigO(n+|E|)$ space complexity and $\bigO(n+\sum_{i\in[n]}d_i^2)$ time complexity, where $d_i$ denotes the degree of vertex $i$ in $G$.
\end{lemma}

 Finally, we note that we use the name \textit{Diamond Percolation} to refer to Algorithm~\ref{alg:commonNeighborsPartitioning} in analogy with local percolation-based clustering, though no percolation threshold is analyzed in this work. 

\subsection{Technical Tools for Studying Algorithm~\ref{alg:commonNeighborsPartitioning}}
\label{subsection:proof_overview}

In this section, we discuss the main tools used to prove the Theorem~\ref{thm:exact},~\ref{thm:almost_exact},~\ref{thm:weak-large} and~\ref{thm:weak}. The full proofs are provided in Appendix~\ref{appendix:proof-arbitrary-partitions}.

Recall that we write $i \rel{T} j$ to indicate that two vertices $i$ and $j$ belong to the same community according to the partition $T$. Given two partitions $C$ and $T$, we say that $C$ is a \emph{refinement} of $T$, denoted $C \preceq T$, if $i\rel{C}j$ implies $i\rel{T}j$ for all $i,j\in[n]$. This condition defines a partial order on the set of partitions. 

To establish that Algorithm~\ref{alg:commonNeighborsPartitioning} recovers the true partition $T_n$, we first show that the partition $C_n$ produced by Algorithm~\ref{alg:commonNeighborsPartitioning} is, with high probability, a refinement of $T_n$. Ensuring that $C_n$ is a refinement of $T_n$ requires the following assumption. We recall that the random variable $S_n$ represents the size of the community to which a uniformly randomly chosen vertex belongs.
%$I_n$ denotes a vertex chosen uniformly at random from $[n]$ and $S_n$ is the size of the community to which~$I_n$ belong, i.e., $S_n=|T_n(I_n)|$. 

\begin{assumption}[Size-sparsity assumption]
\label{assumption:size-sparsity}
  We assume $n^2 \E [S_n^2] q_n^3 p_n^2  = o(1)$ and $q_n=o(n^{-4/5})$. 
\end{assumption}

\begin{theorem}\label{thm:refinement-generalized}
  Let $G_n\sim\PPM(T_n,p_n,q_n)$, such that Assumption~\ref{assumption:size-sparsity} holds. Then, the partition $C_n$ returned by Algorithm~\ref{alg:commonNeighborsPartitioning} satisfies $\P(C_n\preceq T_n)\rightarrow 1$.
\end{theorem}

Assumption~\ref{assumption:size-sparsity} ensures that, with probability tending to 1, any pair of vertices connected by an edge and belonging to different communities has at most one common neighbor. To see this, consider two vertices $i$ and $j$ belonging to different communities, say $a$ and $b$, with respective sizes $s_a$ and $s_b$. The number of common neighbors of $i$ and~$j$ that belong to community $a$ or $b$ is distributed as $\Bin(s_a+s_b-2,p_n q_n)$, and the number of common neighbors that belong to neither $a$ nor $b$ follows $\Bin(n-s_a-s_b, q_n^2)$. Consequently, the probability that $i$ and~$j$ share more than two common neighbors is at most $\bigO((s_a+s_b)^2 p_n^2 q_n^2+ n^2 q_n^4)$. Because there are $\bigO(n^2 q_n)$ pairs of vertices connected by an edge and belonging to different communities, the probability that at least one such pair has more than two common neighbors is vanishing if $n^2 q_n \left( \E [S_n^2] p_n^2 q_n^2  + n^2 q_n^4 \right) = o(1)$. This condition is equivalent to Assumption~\ref{assumption:size-sparsity}. A formal proof is provided in Appendix~\ref{appendix:proof_thm:refinement-generalized}. 

    \cref{assumption:size-sparsity} provides the main motivation for setting the threshold at \emph{two overlapping triangles}. On the one hand, if we were to keep all edges that participate in triangles, then we would have a nonvanishing\footnote{Note that in an \ER graph with density $p_n=\lambda/n$, the number of triangles converges in distribution to $\text{Poi}(\lambda^3/6).$} number of incorrect edges in $G^*$. On the other hand, if we only kept edges participating in three or more triangles, we would have less correct edges in $G^*$, leading to poorer recovery results.

As an example, consider the particular case where $p_n = \Theta(1)$ and $q_n=\Theta(n^{-1})$. Under this setting, Assumption~\ref{assumption:size-sparsity} simplifies to $\E[S_n^2]=o(n)$, a condition that holds for many types of partitions. For instance, in the case of balanced communities of size $s$, \textit{i.e.,} when  $T_n \sim \balanced(n,\lfloor n/s\rfloor,s)$, this condition reduces to $s=o(\sqrt{n})$. Moreover, when $T_n \sim \text{Uniform}(n)$, we have $\E[S_n^2]=\bigO(\log^2 (n))$~\citep{gosgens2024erd} and the condition is automatically satisfied. 

    More generally, whenever $S_n$ is concentrated around its mean, \emph{i.e.}, when $\E[S_n^2]\sim\E[S_n]^2$, then the expected internal degree of a random vertex is $d_{\text{in}}=p_n\E[S_n-1]$, while the expected external degree of a random vertex is $d_{\text{ex}}=q_n\cdot(n-\E[S_n])$, and \cref{assumption:size-sparsity} can be rewritten to $d_{\text{in}}^2d_{\text{ex}}^3\ll n$ and $d_{\text{ex}}\ll n^{1/5}.$ 
    The analysis of community detection algorithms typically requires some upper bound on the external degree and some lower bound on the internal degree. In contrast, we need upper bounds for both internal and external degrees to avoid diamonds between different communities. \cref{section:recovery_criteria} provides lower bounds on $d_{\text{in}}$ that are needed to ensure sufficiently many diamonds inside the communities.
%This is stated in the following theorem. 

Obtaining a refinement of the true communities in itself is neither hard nor informative. For instance, the partition $\{ \{1\}, \{2\}, \cdots, \{n\} \}$, which consists solely of singletons, is a refinement of \textit{any} partition. Therefore, $C_n \preceq T_n$ alone does not guarantee good performance in terms of the correlation coefficient $\rho(C_n,T_n)$ as defined in~\eqref{eq:correlation}. To ensure that $C_n \preceq T_n$ translate into a result involving $\rho(C_n,T_n)$, we require the following assumption on the planted partition $T_n$.
%Finally, we also require an assumption on $m_{T_n}$. 
\begin{assumption}[Concentration of $m_{T_n}$]\label{assumption:concentration_mT}
 For $m_{T_n} = \#\{ij\ :\ i\rel{T_n}j\}$, we assume that $1\ll\E[m_{T_n}]\ll n^2$ and 
    \[
    \frac{m_{T_n}}{\E[m_{T_n}]} \wprto1.
    \]
\end{assumption}
The assumption $\frac{m_{T_n}}{\E[m_{T_n}]} \wprto 1$ holds for many classes of random partitions, including balanced partitions, uniform partitions~\citep{gosgens2024erd}, and the power-law partitions studied in \cref{section:power_law_partitions}.

The following lemma simplifies the asymptotics of the correlation coefficient $\rho(C_n,T_n)$ when $C_n \preceq T_n$. 
\begin{lemma}\label{lem:correlation}
    Let $C_n,T_n$ be sequences of partitions satisfying $\P(C_n\preceq T_n)\to1$ and where $T_n$ satisfies \cref{assumption:concentration_mT}, then
    \[
  \rho(C_n,T_n) - \sqrt{\frac{m_{C_n}}{\E[m_{T_n}]}} \ \stackrel{\P}{\longrightarrow} \ 0.
 \]
\end{lemma}
 \cref{lem:correlation} tells us that when $C_n$ refines $T_n$, the correlation essentially measures the proportion of correctly grouped vertex pairs.

%This lemma is a significant result of our paper and may be of independent interest for future research on the correlation coefficient between planted and predicted partitions. 

\cref{thm:refinement-generalized} and \cref{lem:correlation} are crucial for demonstrating both weak and almost exact recovery.

\section{Recovery of Planted Partitions}
\label{section:arbitrary_partitions}

In this section, we present the conditions for Algorithm~\ref{alg:commonNeighborsPartitioning} to recover a planted partition. Sections~\ref{subsection:exact_recovery}, \ref{subsection:almost_exact_recovery}, and~\ref{subsection:weak_recovery} provide the results and several examples for exact, almost exact, and weak recovery, respectively. All proofs for this section can be found in Appendix~\ref{appendix:proof-arbitrary-partitions}.
%Finally, we present an overview of the proofs in Section~\ref{subsection:proof_overview}.

Throughout this section, $G_n\sim\PPM(T_n,p_n,q_n)$ is a PPM with vertex set $[n]$, planted partition~$T_n$, internal connection probability $p_n$, and external connection probability~$q_n$.

\subsection{Exact Recovery}
\label{subsection:exact_recovery}

 To derive a consistency result for exact recovery, we impose a lower bound on the size of the smallest non-singleton community. 
 %Some results require an upper-bound on the size of the smallest community. Specifically, for exact recovery, we impose Assumption~\ref{assumption:minimum_size}, while for almost exact recovery we impose Assumption~\ref{assumption:minimum_size_soft}, a slightly less restrictive condition.  

\begin{assumption}[Minimum community size]
\label{assumption:minimum_size}
 There exists some sequence $s_n^{(\min)}\to\infty$ so that 
    \[
    \P \left( \exists i \in [n] \colon \ 1 \, < \, |T_n(i)| \, < \, s_n^{(\min)} \right) \rightarrow 0.
    \]
\end{assumption}

\cref{assumption:minimum_size} ensures that no community becomes disconnected, which would otherwise make exact recovery impossible. 
The following theorem states that if every community is sufficiently large and has enough internal edges, the algorithm reconstructs the true partition exactly. 

\begin{theorem}
\label{thm:exact} Consider a graph $G_n\sim\text{PPM}(T_n,p_n,q_n)$, where the sequence of random partitions $T_n$ satisfy Assumptions~\ref{assumption:size-sparsity} and~\ref{assumption:minimum_size}, and 
\[
p_n=\sqrt{\frac{\log\E[m_T]+\log\log\E[m_T]+\omega(1)}{s_n^{(\min)}}}.
\]
%the probability $p_n$ satisfies 
% \[
%  \E[m_T]\cdot s_n^{(\min)}p_n^2e^{-p_n^2s_n^{(\min)}}\to0.
% \]
 Then Algorithm~\ref{alg:commonNeighborsPartitioning} achieves exact recovery.
\end{theorem}

 %We prove Theorem~\ref{thm:exact} in Section~\ref{}. %We provide an overview of the proof of Theorem~\ref{thm:exact} in Section~\ref{subsection:proof_overview}. 
 To compare \cref{thm:exact} with existing results in the literature, we present some examples of its application.

\begin{example}
\label{example:exact_recovery_balanced}
    Consider $G_n\sim\PPM(T_n,p_n,q_n)$, where $T_n\sim\balanced(n,k,s_n)$, for $k$ fixed, $1\ll s_n\ll n^{2/3}$, $q_n=\bigO(n^{-1})$ and $p_n\ge\sqrt{3s_n^{-1}\log s_n}$. Then Algorithm~\ref{alg:commonNeighborsPartitioning} achieves exact recovery. 
\end{example}
 \citet{chen2016statistical} establishes exact recovery of a single community of size $s_n = \Omega(\log n)$ (that is, $T_n\sim\balanced(n,1,s_n)$), while the previous example allows for one or more communities having a much smaller size ($1\ll s_n\ll\log n$). 

%\MDnote{I believe \cite[Theorem~10]{chen2016statistical} imposes $s p^2 \ge c_2 \log n$ when $q = \bigO(n^{-1})$ for success of degree thresholding.} 
%\MDnote{For success of SDP, \cite[Theorem~6]{chen2016statistical} imposes $sp \gesim \log n$ when $q = \bigO(n^{-1})$.}

\begin{example}
\label{example:exact_recovery_same_size}
    Consider $G_n\sim\PPM(T_n,p_n,q_n)$, where $T_n\sim\balanced(n,\lfloor n/s_n\rfloor,s_n)$, where $\log n+3\log\log n\le s_n\ll \sqrt{n}$, $q_n=\bigO(n^{-1})$ and $p_n\ge\sqrt{s_n^{-1}(\log n+3\log s_n)}$. Then Algorithm~\ref{alg:commonNeighborsPartitioning} achieves exact recovery. The condition $s_n\ge \log n+3\log\log n$ is required to ensure $\sqrt{s_n^{-1}(\log n+3\log s_n)}\le1$. 
\end{example}

Let us compare Example~\ref{example:exact_recovery_same_size} with results established in \citet{chen2016statistical}. First, consider the case $s_n = \alpha \log n$. According to \cite[Theorem~10]{chen2016statistical}, a simple degree thresholding approach succeeds at exact recovery if $p \ge c \alpha^{-1/2}$ for some unspecified constant $c>0$. However, this condition may never be satisfied if the unspecified $c$ is too large. In fact, by scrutinizing the proof of \cite[Theorem~10]{chen2016statistical}, we observe that $c \ge 144$ is needed. In contrast, our result provides an explicit lower-bound on $p_n$ to guarantee the exact recovery by Algorithm~\ref{alg:commonNeighborsPartitioning}. More generally, for $\log n \ll s \ll n^{1/2}$, \cite[Theorem~10]{chen2016statistical} requires $p_n \ge c \sqrt{ \frac{\log n}{s_n} }$ (again with $c \ge 144$). Hence, the condition $p_n \ge \sqrt{\frac{\log n + 3 \log s_n}{s_n} }$ in Example~\ref{example:exact_recovery_same_size} is strictly less restrictive. Finally, \cite[Theorem~6]{chen2016statistical} shows that a convex relaxation of MLE achieves exact recovery if $s_n p_n \ge c \log n$. Again, this requires $s_n\ge c\log n$, which is more restrictive than our requirement. 

%Then, Corollary~5 (resp., Theorem~6) in \cite{chen2016statistical} show that the maximum likelihood estimator (MLE) (resp., a convex relaxation of the MLE) achieve exact recovery under the additional condition $\alpha p \ge c_1$ (resp., $\alpha p \ge c_2$) for some unspecified constant $c_1, c_2 > 0$. In contrast, our result allows for $p \ll 1$. 

\begin{example}
    Consider a planted partition $T_n$ consisting of an arbitrary number of communities whose sizes are in the range $[s_n^{(\min)},n^\alpha]$ for $s_n^{(\min)}\gg\log n$ and $\alpha<1/2$. We have $m_T\le \tfrac{1}{2}n^{1+\alpha}$. Thus, $\cD$ achieves exact recovery for $q_n=\bigO(n^{-1})$ and
    \[
    p_n\wge\sqrt{\frac{(1+\alpha)\log n+o(\log n)}{s_n^{(\min)}}}.
    \]
\end{example}

This last example highlights that the large communities only increase the threshold by a constant factor. The condition $s_n^{(\min)} p_n^2 \gesim \log n$ is analogous to the condition $s_n p_n^2 \gesim \log n$ obtained in Example~\ref{example:exact_recovery_same_size}. Notably, this result is new to the literature, as \citet{chen2016statistical} focuses exclusively on communities of equal size. Finally, this result is consistent with the fact that the exact recovery threshold in the PPM (with a finite number of communities) is primarily determined by the difficulty of recovering the smallest community, as this is the most challenging community to identify.

\begin{example}
    Consider an \ER random graph $G_n$ with connection probability $q_n=o(n^{-4/5})$, or equivalently, $G_n\sim\PPM(T_n,p_n,q_n)$, where $T_n$ consists of $n$ singleton communities and $p_n \in [0,1]$ is arbitrary. Then Algorithm~\ref{alg:commonNeighborsPartitioning} achieves exact recovery. That is, Algorithm~\ref{alg:commonNeighborsPartitioning} correctly detects the absence of communities in $G_n$.
\end{example}

This last example highlights that our algorithm does not lead to false positives in \ER random graphs, as long as the graph is not too dense.

\subsection{Almost Exact Recovery}
\label{subsection:almost_exact_recovery}
While almost exact recovery has been studied in the case of $k=\Theta(1)$, it has (to the  best of our knowledge) not been studied for $k$ growing arbitrarily fast and in the presence of arbitrarily small communities. Therefore, the results in this section are the first results on almost exact recovery of small communities.
Similar to exact recovery, we impose a constraint on the number of small communities. 

\begin{assumption}[Soft minimum community size]
\label{assumption:minimum_size_soft}
 There exists some sequence $s_n^{(\min)}\to\infty$ so that 
\[
\frac{\mathbb E[(S_n-1)\cdot\mathbb 1\{S_n<s_n^{(\min)}\}]}{\mathbb E[S_n-1]}\to 0.
\]
\end{assumption}
\cref{assumption:minimum_size_soft} is less restrictive than \cref{assumption:minimum_size}. Specifically, \cref{assumption:minimum_size_soft} ensures that communities of size smaller than $s_n^{(\min)}$ have negligible contribution to the expectation of $\E[S_n-1]$. Small communities are still admissible and the algorithm may not correctly recover them, but it turns out that they do not hinder almost exact recovery in terms of the correlation coefficient $\rho$. 
If we denote the size-biased version of $S_n-1$ by $(S_n-1)^*$, then \cref{assumption:minimum_size_soft} is equivalent to $(S_n-1)^*\stackrel\P\to\infty$. Since $(S_n-1)^*$ stochastically dominates $S_n-1$, a sufficient condition for \cref{assumption:minimum_size_soft} is $S_n\stackrel\P\to\infty$. 

\begin{theorem}
\label{thm:almost_exact}
Consider a graph $G_n\sim\text{PPM}(T_n,p_n,q_n)$, where the sequence of random partitions $T_n$ satisfy Assumptions~\ref{assumption:size-sparsity} and~\ref{assumption:concentration_mT}, and Assumption~\ref{assumption:minimum_size_soft} with $s_n^{(\min)}\to\infty$ so that
\[
p_n=\sqrt{\frac{2\log s_n^{(\min)} +\log\log s_n^{(\min)} +\omega(1)}{s_n^{(\min)}}}.
\]
%the probabilities $p_n$ satisfy
% \[
% \left(s_n^{(\min)}\right)^3p_n^2e^{-p_n^2s_n^{(\min)}}\to0.
% \]
 Then Algorithm~\ref{alg:commonNeighborsPartitioning} achieves almost exact recovery.
\end{theorem}

% We provide an overview of the proof of Theorem~\ref{thm:almost_exact} in Section~\ref{subsection:proof_overview}. %Let us state some examples of applications of Theorem~\ref{thm:almost_exact}.

 Recall that when Assumptions~\ref{assumption:size-sparsity} and~\ref{assumption:concentration_mT} hold, \cref{thm:refinement-generalized} and \cref{lem:correlation} ensure that $\rho(C_n,T_n)^2-\frac{m_{C_n}}{\E[m_{T_n}]}\stackrel{\P}{\rightarrow}0$. Thus, $\rho(C_n,T_n) \wprto 1$ holds whenever $\tfrac{\E[m_{C_n}]}{\E[m_{T_n}]} \to 1$. The proof of this result establishes \cref{thm:almost_exact}. 
 
%\begin{remark}\label{rem:almost-exact}
%    If Assumptions~\ref{assumption:size-sparsity} and~\ref{assumption:concentration_mT} hold, and additionally $\E[m_C]\sim\E[m_T]$, then $\rho(C,T)\wprto1$. Hence, almost exact recovery is achieved when $\E[m_C]\sim\E[m_T]$.
%\end{remark}

Comparing~\cref{thm:almost_exact} to~\cref{thm:exact}, we see that the $\E[m_T]$ is replaced with $\left(s_n^{(\min)}\right)^2$. This reflects the fact that our proof is based on ensuring that for an arbitrary community of size at least $s_n^{(\min)}$, all vertex pairs inside that community have at least two common neighbors w.h.p., so that $\cD$ recovers it correctly.

 The following \cref{example:almost_exact_balancedPartition} shows that Algorithm~\ref{alg:commonNeighborsPartitioning} achieves almost exact recovery for balanced partitions, where each community has size $s_n \gg 1$, while \cref{example:exact_recovery_same_size} requires $s_n = \Omega( \log n )$ for exact recovery. 

\begin{example}
\label{example:almost_exact_balancedPartition}
  Let $T_n\sim\balanced(n,k_n,s_n)$ for $1\ll s_n\ll\sqrt{n}$ and $k_ns_n\le n$, and suppose $p_n\ge\sqrt{\tfrac{3\log s_n}{s_n}}$ and $q_n=o(n^{-4/5})$. Then Algorithm~\ref{alg:commonNeighborsPartitioning} achieves almost exact recovery.
\end{example}

Note that the $-1$ in the denominator of \cref{assumption:minimum_size_soft} is necessary to allow for cases where $\E[S_n]\to1$, like \cref{example:almost_exact_balancedPartition} with $k_ns_n^2\ll n$. For example, $k_n=1$ and $s_n\sim\log\log n$.

\begin{example}\label{ex:uniform}
 Suppose $T_n$ is drawn uniformly from the set of all partitions of $[n]$. We recall that in that case, we have $S_n/\log n\stackrel{\P}{\rightarrow}1$ and $m_{T_n} / \E[m_{T_n}] \wprto 1$ \citep{gosgens2024erd}. Hence, Algorithm~\ref{alg:commonNeighborsPartitioning} achieves almost exact recovery for $p_n=\omega\left(\sqrt{\tfrac{\log\log n}{\log n}}\right)$ and $q_n=o(n^{-4/5})$.
\end{example}

\subsection{Weak Recovery}
\label{subsection:weak_recovery}
In this section, we will present two different recovery results for weak recovery. One tailored to partitions where a relatively large fraction of vertices is part of large communities, while the other is tailored to the case where almost all vertices are part of communities of bounded size.
Note that weak recovery is trivially implied by almost exact recovery. Therefore, we will mostly focus on settings where \cref{assumption:minimum_size_soft} does \emph{not} hold, since otherwise we already achieve almost exact recovery. Whenever \cref{assumption:minimum_size_soft} does not hold, then for every $s_n\to\infty$, we have
\[
\varepsilon_n(s_n)=\frac{\mathbb E[(S_n-1)\cdot\mathbb 1\{S_n<s_n\}]}{\mathbb E[S_n-1]}\not\to 0.
\]
However, $\varepsilon_n(s_n)\in[0,1]$, so that by the Bolzano-Weierstrass theorem, it must converge along some subsequence. That is, there exists some $\varepsilon\in[0,1]$ and some subsequence $n_\ell$ so that
$
\varepsilon_{n_\ell}(s_{n_\ell})\to\varepsilon$,
as $\ell\to\infty$.
It is easy to construct pathological partitions $T_n$ and sequences $s_n\to\infty$ for which $\varepsilon_n(s_n)$ oscillates, but these are not of theoretic interest. 
Therefore, we will focus on cases where there exists some sequence $s_n\to\infty$ for which $\varepsilon_n(s_n)$ converges, so that it converges along every subsequence.

\begin{assumption}[Weak minimum community size]
    \label{assumption:minimum_size_weak}
    There exists some sequence $s_n^{(\min)}\to\infty$ and $\varepsilon\in(0,1)$ so that
    \[
\frac{\mathbb E[(S_n-1)\cdot\mathbb 1\{S_n<s_n^{(\min)}\}]}{\mathbb E[S_n-1]}\to\varepsilon.
\]
\end{assumption}

We present the following recovery result:

\begin{theorem}\label{thm:weak-large}
    Consider a graph $G_n\sim\text{PPM}(T_n,p_n,q_n)$, where the sequence of random partitions $T_n$ satisfy Assumptions~\ref{assumption:size-sparsity} and~\ref{assumption:concentration_mT}, and Assumption~\ref{assumption:minimum_size_weak} with $s_n^{(\min)}\to\infty$ and $\varepsilon>0$ so that 
\[
p_n=\sqrt{\frac{2\log s_n^{(\min)} +\log\log s_n^{(\min)} +\omega(1)}{s_n^{(\min)}}}.
\]
Then Algorithm~\ref{alg:commonNeighborsPartitioning} achieves weak recovery. More precisely,
\[
\rho(\cD(G_n),T_n)\ge\sqrt{1-\varepsilon}+o_\P(1).
\]
\end{theorem}
\begin{proof}
    We make use of \cref{lem:correlation} and will prove that
    \[
        \frac{m_{C_n}}{\E[m_{T_n}]}\ge 1-\varepsilon+o_\P(1).
    \]
    We separate $m_{T_n}=m_{T_n}^-+m_{T_n}^+$ into vertex pairs that are part of communities of size smaller than $s_n^{(\min)}$ and communities of size at least $s_n^{(\min)}$. By \cref{assumption:minimum_size_weak}, we have $\E[m_{T_n}^-]\sim\varepsilon \E[m_T]$ and $\E[m_{T_n}^+]\sim(1-\varepsilon) \E[m_T]$. We similarly separate $m_{C_n}=m_{C_n}^-+m_{C_n}^+$. Following the same steps as in the proof of \cref{thm:almost_exact}, we conclude that $m_{C_n}^+/\E[m_{T_n}^+]=1+o_\P(1)$.
    We will bound $m_{C_n}^-\ge0$ and write
    \begin{align*}
        \frac{m_{C_n}^-+m_{C_n}^+}{\E[m_{T_n}]}\ge (1-\varepsilon)\frac{m_{C_n}^+}{(1-\varepsilon)\E[m_{T_n}]}=1-\varepsilon+o_\P(1),
    \end{align*}
    which completes the proof.
\end{proof}
Comparing \cref{thm:weak-large} to \cref{thm:almost_exact}, we see that we merely replace \cref{assumption:minimum_size_soft} with~\cref{assumption:minimum_size_weak}. This allows us to ensure that sufficiently many communities are completely recovered, so that a nonvanishing fraction of the vertex pairs is correctly recovered. 

Next, we develop a weak recovery result that allows for communities of bounded size.
If there is some $s_n\to\infty$ for which $\varepsilon_n(s_n)$ converges, but \cref{assumption:minimum_size_soft} and \cref{assumption:minimum_size_weak} do not hold, then it must hold that for every $s_n\to\infty$, we have $\varepsilon_n(s_n)\to1$. The following lemma shows that this implies that $S_n$ has a convergent subsequence:

\begin{lemma}\label{lem:uniform-integrability}
    If for every $s_n\to\infty$, it holds that
    \[
\frac{\mathbb E[(S_n-1)\cdot\mathbb 1\{S_n<s_n\}]}{\mathbb E[S_n-1]}\to1,
\]
then the sequence of conditional distributions $(S_n\ |\ S_n>1)$ is uniformly integrable and there exists some limiting random variable $S$ and a subsequence $n_\ell$ so that $(S_{n_\ell}\ |\ S_{n_\ell}>1)\stackrel d\to S$ as $\ell\to\infty$. Moreover, $\E[S_{n_\ell}\ |\ S_{n_\ell}>1]\to\E[S]$.
\end{lemma}

As we are mainly interested in the convergent subsequence, we will impose the following assumption. 
\begin{assumption}[Convergent size distribution]\label{assumption:convergent-s}
 There exists a distribution $S$ with $\E[S]<\infty$ and $\P(S\ge4)>0$, and such that 
 $(S_n\ |\ S_n>1)\stackrel d\to S$ and $\cE{S_n}{S_n>1}\to\E[S].$
\end{assumption}

\begin{theorem}\label{thm:weak}
Consider a graph $G_n\sim\text{PPM}(T_n, p, q_n)$, where $p>0$, $T_n$ satisfies Assumptions~\ref{assumption:size-sparsity},~\ref{assumption:concentration_mT} and~\ref{assumption:convergent-s}. Then Algorithm~\ref{alg:commonNeighborsPartitioning} achieves weak recovery. 
\end{theorem}
We first discuss the conditions of \cref{thm:weak}. Note that the condition $\P(S\ge4)>0$ is necessary because Algorithm~\ref{alg:commonNeighborsPartitioning} is based on finding subgraphs of size~$4$.
Regarding $p_n=p>0$, note that for a partition with balanced communities of size $s<\infty$, weak recovery is not feasible if $p_n\cdot s\to0$, as this means that a typical vertex will not have any connections to its community. Since $\E[S_n]$ is bounded, $p_n$ must be nonvanishing, which is why we consider constant $p_n=p$ in~\cref{thm:weak}.
Moreover, \cref{assumption:size-sparsity} is not implied by \cref{assumption:convergent-s}:
suppose that $q_n=\bigO(n^{-1})$ and consider a single community of size $n^{3/4}$ while the remainder of communities are of size $2$. Then $S_n\stackrel d\to 2$. Since $p > 0$ so that \cref{assumption:size-sparsity} simplifies to $\E[S_n^2]=o(n)$, but we have $\E[S_n^2]\sim n^{5/4}$. %This condition is required to ensure with high probability that the detected community will be a refinement of the true communities. 

The proof of \cref{thm:weak} provides a lower bound for $\rho(C_n,T_n)$, where $C_n$ is the partition obtained by Algorithm~\ref{alg:commonNeighborsPartitioning}. More precisely, we establish that 
    \begin{align}
    \label{eq:lower_bound_weak_recovery}
        \rho(C_n,T_n) \wge \sqrt{\frac{\E_{H\sim \text{ER}(S,p)}[|C^{(H)}(1)|-1]}{\E[S-1]}}+o_\P(1),
    \end{align}
    where $C^{(H)}=\cD(H)$ (that is, we apply Algorithm~\ref{alg:commonNeighborsPartitioning} to an \Erdos-\Renyi random graph with $S$ vertices and connection probability $p$ to obtain $C^{(H)}$, and $|C^{(H)}(1)|$ is the number of vertices in the detected community of vertex $1$). 
    The quantity in the right hand side of~\eqref{eq:lower_bound_weak_recovery} is positive. To see this, note that if $S=4$, then $|C^{(H)}(1)|=4$ if this community forms a clique, which occurs with probability $p^{4\choose 2}=p^6$. For $S>4$, $\P(|C^{(H)}(1)|\ge4)\ge p^6$, as we can bound this probability by the probability that $1$ forms a clique with vertices $2,3$ and $4$. For $S<4$, we use the bound $|C^{(H)}(1)|\ge1$.
    We conclude that the given lower bound is asymptotically at least
    \[
    \frac{3p^6\P(S\ge4)}{\E[S]-1} \ > \ 0.
    \]
    The following example demonstrates that the more precise performance bound from~\eqref{eq:lower_bound_weak_recovery} is quite sharp when compared to the actual performance of Algorithm~\ref{alg:commonNeighborsPartitioning} on PPM graphs:

\begin{example}
    Suppose that $T_n\sim\balanced(n,k_n,s)$ for $s\ge4$, $k_ns\le n$ and $k_n\to\infty$. Then Algorithm~\ref{alg:commonNeighborsPartitioning} achieves weak recovery. Moreover, we have $\rho(C_n,T_n) \ge \Delta +o_\P(1) $, where
    \begin{equation}\label{eq:weak-balanced}
    \Delta \weq \sqrt{\frac{ s' - 1 }{ s - 1 }}
    \quad \text{ for } \quad 
    s' \weq \E_{H\sim ER(s,p)}[|C^{(H)}(1)|]
    \end{equation}
    The quantity $\Delta$ provides a lower bound for the asymptotic performance of Algorithm~\ref{alg:commonNeighborsPartitioning}. However, obtaining closed-form expressions for $s'$ is challenging, even in the special case of equal-size communities. Instead, we can efficiently estimate this expectation by (i) sampling several \Erdos-\Renyi graphs with $s$ vertices and edge connection probability $p$, (ii) applying Algorithm~\ref{alg:commonNeighborsPartitioning} to these graphs, and (iii) computing the empirical average. In \cref{fig:scores-p}, we use this approach to estimate $\Delta$ for various values of $p$ and $s$. We observe that $\Delta$ rapidly approaches $1$ as $s$ increases. In \cref{fig:scores}, we show that the empirical performance of Algorithm~\ref{alg:commonNeighborsPartitioning} closely aligns with the asymptotic bound given by the estimated value of~$\Delta$.     
\end{example}

\begin{figure}[!ht]
\centering
\begin{subfigure}[b]{0.48\linewidth}
    \includegraphics[width=\linewidth]{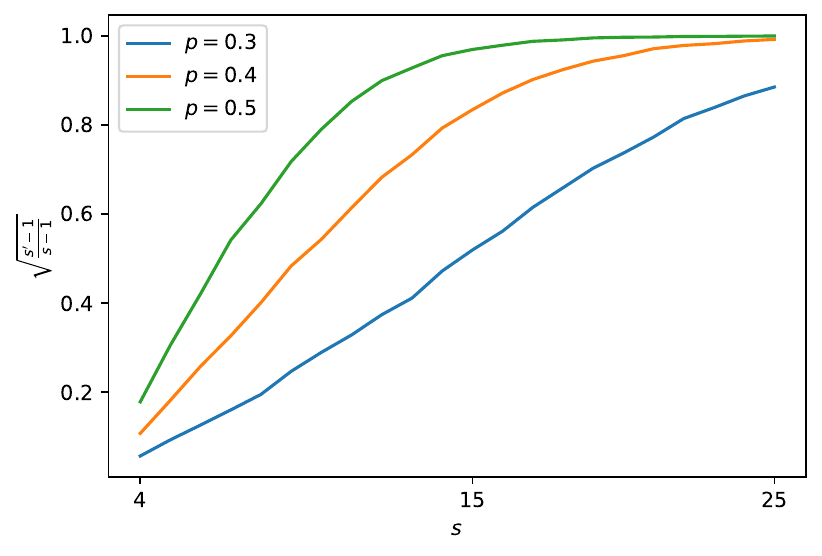}
    \caption{$\sqrt{\tfrac{s'-1}{s-1}}$}
 \label{fig:scores-p}
 \end{subfigure} \hfill
\begin{subfigure}[b]{0.48\linewidth}
    \includegraphics[width=\linewidth]{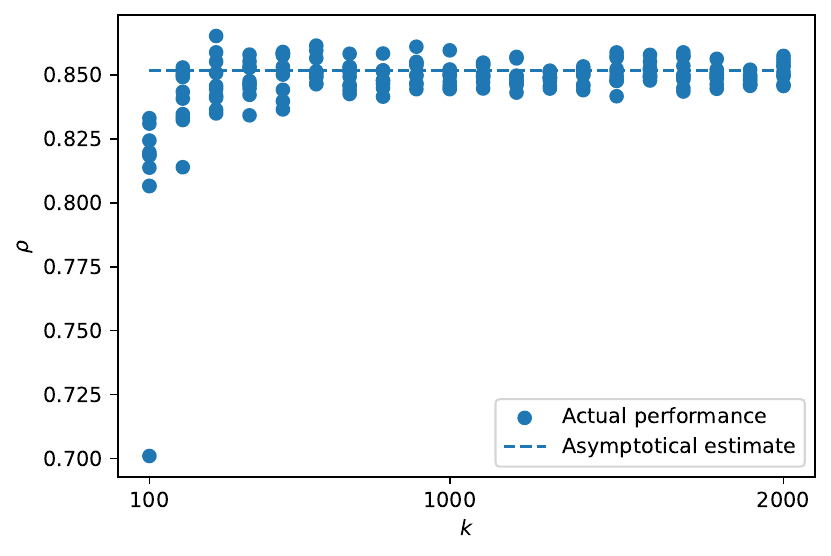}
    \caption{$\rho(C_n, T_n)$}
    \label{fig:scores}
\end{subfigure} 
    \caption{Figure~\ref{fig:scores-p}: Estimation of the quantity $\Delta$ defined in~\eqref{eq:weak-balanced}. For each estimate, we sample $5000$ random graphs from $ER(s,p)$ and apply Algorithm~\ref{alg:commonNeighborsPartitioning} to each of them. \\ 
    %Estimation of the asymptotic performance $\sqrt{\tfrac{s'-1}{s-1}}$ from~\eqref{eq:weak-balanced}, where $s'=\E_{H\sim ER(s,p)}[|C^{(H)}(1)|]$, for different values of $p$ and $s$. For each estimate, we sample $5000$ random graphs from $ER(s,p)$ and apply Algorithm~\ref{alg:commonNeighborsPartitioning} to each of them. \\
    Figure~\ref{fig:scores}: Comparison of the performance of Algorithm~\ref{alg:commonNeighborsPartitioning} to the estimated asymptotic performance established in~\eqref{eq:weak-balanced}, when $T_{n}\sim\balanced(k\cdot s,k,s)$ with $p=0.5$, $s=11$, and $q=5/(k\cdot s-s)$ (so that in expectation, every vertex has five neighbors inside and outside its community).}
\end{figure}

\section{Recovery of Power-law Partitions}
\label{section:power_law_partitions}

In this section, we focus on power-law partitions. We recall some results on power-law random variables in \cref{subsec:power-law-rv}, show how to construct power-law partitions in \cref{subsec:construction_power_law_partition}, and state the results for recovering power-law partitions in \cref{subsec:recovery_power_law}.

\subsection{Power-law Random Variables}
\label{subsec:power-law-rv}
It has been observed in many real-world networks that the community sizes follow a power law~\citep{lancichinetti2008benchmark,stegehuis2016power,voitalov2019scale}. Informally, this means that the probability of observing a community of size larger than $s$ scales like $s^{1-\tau}$ for some $\tau>1$. In our setting, we formalize this using the following definition:

\begin{definition}
    A random partition $T_n$ asymptotically follows a power law with exponent $\tau>1$ if there is some scaling sequence $s_n$ so that
    \[
\frac{S_n}{s_n} \ \stackrel d{\to} \ S,
\]
for some random variable $S$ that satisfies 
\[
\P(S\ge x) \weq \Theta(x^{1-\tau}).
\]
\end{definition}
The random variable $S$ in the above definition is said to follow a power-law distribution with exponent~$\tau$. Note that other works, such as~\citet{voitalov2019scale}, use a more general definition in which the $\Theta(x^{1-\tau})$ is replaced by regularly-varying functions. For simplicity, we adopt the narrower definition given above.

If $S$ follows a power-law distribution with exponent $\tau>2$, then $\E[S^k]<\infty$ if $k<\tau-1$ and $\E[S^k]=\infty$ if $k>\tau-1$. The simplest example of a probability distribution that satisfies a power-law is the \emph{Pareto distribution} $Z\sim\text{Pareto}(c,\beta)$, given by 
\[
 \P(Z>z) \weq (c/z)^{\beta},
\]
where $c>0$ is the scale parameter and $\beta>0$ is the tail exponent. The Pareto distribution follows a power law with exponent $\tau=\beta+1$.
%We write $Z\sim\text{Pareto}(c,\beta)$ and say that $Z$ follows the \emph{Pareto} distribution with scale $c>0$ and shape $\beta>0$ if
%\[\P(Z> z)=\min\{1,(c/z)^\beta\}.\]
%The Pareto distribution is also related to the exponential distribution: if $X\sim \text{Exp}(1)$, then $c\cdot e^{X/\beta}\sim\text{Pareto}(c,\beta)$. 

%This way, we can view the distribution $\text{Prop}(\tau,k)$ as a normalized Pareto distribution.

\subsection{Construction of Power-law Partitions}
\label{subsec:construction_power_law_partition}

%This section presents several key properties about the distribution of power-law partitions. 
Recall from \cref{subsection:arbitrary_partitions} that, if $T_n\sim\text{Powerlaw}(\tau,k,n)$ for $\tau>2$ and $k\in[n]$, then each vertex is assigned to the community $a\in[k]$ with probability
\begin{align}
\label{eq:def_Pi_a}
 \Pi_a \weq \frac{e^{X_a/\tau}}{\sum_{b\in[k]}e^{X_b/\tau}},
\end{align}
where $(X_a)_{a\in[k]}$ is a sequence of i.i.d. exponentially distributed random variables with parameter $1$.
We refer to $\Pi_a$ as the \emph{proportion} of community $a$ and denote its distribution by $\Pi_a\sim\text{Prop}(\tau,k)$.
For $T_n\sim\text{Powerlaw}(\tau,k,n)$, let $\Pi^*_n$ denote the proportion of the community of a vertex chosen uniformly at random. The distribution of $\Pi^*_n$ corresponds to the \emph{size-biased} distribution of $\Pi_a$~\citep{arratia2019size}. That is, given proportions $\Pi_a=\pi_a$ for $a\in[k]$, we have 
\[
\cP{\Pi^*_n=x}{\forall a\in[k] \colon \Pi_a=\pi_a} \weq x\cdot \left| \{a\in[k] \colon \pi_a=x\} \right|.
\]
Because this distribution does not directly\footnote{It depends on $k$, which may depend on $n$.} depend on $n$, we will abbreviate $\Pi^*=\Pi^*_n$.
The following theorem states that this construction of partitions leads to power-law distributed community sizes. 

\begin{theorem}\label{thm:power-law-rescaled}
    Let $\tau>2$ and $1\ll k_n\ll n$. If $T_n\sim\text{Powerlaw}(\tau,k_n,n)$, then 
    \[
        \frac{S_n}{n/k_n} \ \stackrel d\to \ \text{Pareto}\left(1-\tfrac1\tau, \tau-1\right),
    \]
    In particular, $S_n$ asymptotically follows a power law with exponent $\tau$ and scaling $s_n=n/k_n$.
\end{theorem}
When $k_n=\Theta(n)$, we can determine the limiting distribution of $S_n$ exactly:
%We can in fact be more precise about the distribution of $S_n$. Although the following lemma is not used later in the manuscript, we incorporate it for future reference. 

\begin{lemma}\label{lem:powerlaw-mixed-Poisson}
    Let $\tau>2$ and $k_n\sim n/s$, for $s>1$. If $T_n\sim\text{Powerlaw}(\tau,k_n,n)$, then \[
        \P(S_n=r+1)\to \E\left[\frac{Z^r}{r!}e^{-Z}\right],
    \]
    where $Z\sim\text{Pareto}(s\cdot(1-\tfrac1\tau),\tau-1)$. That is, $S_n-1$ converges in distribution to a \emph{mixed Poisson distribution} with Pareto mixture, so that $S_n$ asymptotically follows a power law with exponent $\tau$.
\end{lemma}

\subsection{Recovery of Power-law Partitions}
\label{subsec:recovery_power_law}

In this section, we apply the results of \cref{section:arbitrary_partitions} about the recovery of planted partitions to show that Algorithm~\ref{alg:commonNeighborsPartitioning} recovers power-law partitions.

\begin{corollary}[Recovery of power-law partitions]
\label{cor:recovery_power_law}
Let $G_n\sim\text{PPM}(T_n,p_n,q_n)$ where $T_n$ is sampled from $\text{Powerlaw}(\tau,k_n,n)$ with $\tau>2$ and $q_n = \bigO(n^{-1})$. 
\begin{enumerate}
  \item Suppose that   
  \[
  \max\{\sqrt{n},n^{\frac{1}{\tau-1}}\}\ll k_n \le \frac{\varepsilon^2}{4} \frac{\tau-1}{\tau} \frac{n}{\log n} 
  \quad \text{ and } \quad  
  p_n^2 \wge \frac{3\tau}{\tau-1-\varepsilon}\frac{k_n \log n}{n},
 \]
 for some $\varepsilon>0$. Then Algorithm~\ref{alg:commonNeighborsPartitioning} achieves exact recovery. 
 \item Suppose that  
 \[
  \max\{\sqrt{n},n^{\frac{1}{\tau-1}}\}\ll k_n \ll n 
  \quad \text{ and } \quad 
  p_n^2 \wge \frac{3 \tau}{ \tau - 1} \frac{ k_n \log(n/k_n) }{ n }. 
  %\frac{ 3 k_n \log(n/k_n) }{ (1-\tfrac1\tau)n }.
 \]
 Then Algorithm~\ref{alg:commonNeighborsPartitioning} achieves almost exact recovery. 
 \item Suppose that $p> 0$ and $k_n\sim n/s$ for $s>1$. Then Algorithm~\ref{alg:commonNeighborsPartitioning} achieves weak recovery. 
\end{enumerate} 
\end{corollary}

\begin{proof}
(i) To establish that Algorithm~\ref{alg:commonNeighborsPartitioning} achieves exact recovery of $T_n$, we show that the conditions of Theorem~\ref{thm:exact} are satisfied. 
 \cref{assumption:size-sparsity} holds since we assumed $q_n = \bigO(n^{-1})$ while \cref{lem:powerlaw-condition} in Appendix~\ref{subsection:additional_lemmas_power-law} shows that $\E [S_n^2] = o(n)$. \cref{lem:powerlaw-minsize} in
Appendix~\ref{subsection:powerlaw-minsize} shows that with high probability, all communities are larger than $s_n^{(\min)}=\tfrac{(\tau-1-\varepsilon)n}{\tau k}$. Hence  \cref{assumption:minimum_size} is satisfied. Finally, the assumption on $p_n$ ensures that the bound in \cref{thm:exact} is satisfied, which completes the proof. 

(ii)  Similarly, to establish almost exact recovery, we prove that the assumptions of \cref{thm:almost_exact} are satisfied. Again, \cref{assumption:size-sparsity} follows from our assumption on $q_n$ and \cref{lem:powerlaw-condition}. Moreover, \cref{lem:powerlaw-mT} establishes \cref{assumption:concentration_mT}. \cref{thm:power-law-rescaled} shows that $\P(k_nS_n/n\ge1-\tau^{-1})\to1$, so that \cref{assumption:minimum_size_soft} is satisfied with $s_n^{(\min)}=(1-\tau^{-1})\tfrac{n}{k_n}$. Finally, the assumption on $p_n$ ensures that the bound in \cref{thm:almost_exact} is satisfied.

(iii)  Finally, to prove weak recovery, we show that the assumptions of Theorem~\ref{thm:weak} are satisfied. Assumptions~\ref{assumption:size-sparsity} and~\ref{assumption:concentration_mT} are implied by the assumption on $q_n$, \cref{lem:powerlaw-condition}, and \cref{lem:powerlaw-mT}.
    \cref{lem:powerlaw-mixed-Poisson} tells us that $S_n-1$ converges in distribution to a mixed Poisson with Pareto mixture. This random variable has a finite mean. Hence, the distribution conditioned on $S_n>1$ must also converge to a random variable with finite mean. Additionally, $\P(S_n\ge4)$ has a positive limit. What remains to show is that the expectation of $\E[S_n-1]$ converges to the expectation of our mixed Poisson random variable. We write $\E[S_n-1]=\tfrac2n\E[m_T]$ and use \cref{lem:powerlaw-mT} to conclude that
    \[
    \E[S_n-1] \wsim \frac{n(\tau-1)^2}{k\tau(\tau-2)} \ \to \ s\cdot \frac{(\tau-1)^2}{\tau(\tau-2)}.
    \]
    The expectation of a mixed Poisson random variable is equal to the expectation of the mixture distribution. The expectation of $Z\sim\text{Pareto}(c,\beta)$ is $\E[Z]=\frac{\beta\cdot c}{\beta-1}$. Substituting $\beta=\tau-1$ and $c=s\cdot (1-\tau^{-1})$ yields
    \[
    \E[Z] \weq s\cdot\frac{(1-\tau^{-1})(\tau-1)}{\tau-2} \weq s\cdot \frac{(\tau-1)^2}{\tau(\tau-2)}.
    \]
    This tells us that $\E[S_n-1]$ indeed converges to the expectation of the Poisson mixture. Therefore, this also holds after conditioning on $S_n>1$ and $S>1$. We conclude that the conditions of \cref{thm:weak} are satisfied.

\end{proof}

\section{Experiments}\label{section:experiments}

 In this section, we evaluate the empirical performance of the Diamond Percolation method on random graphs and compare it with other community detection algorithms. The code for these experiments is available on GitHub\footnote{See \href{https://github.com/MartijnGosgens/DetectingSmallCommunities.git}{https://github.com/MartijnGosgens/DetectingSmallCommunities}.}.
 We illustrate the statements (1) and (2) of \cref{cor:recovery_power_law} in Figures~\ref{fig:almost-exact} and~\ref{fig:exact}, respectively. Specifically, we sample graphs $G_n \sim \PPM(T_n,p,q_n)$ with parameters $p=0.4$, $q_n=5/n$, and partition $T_n \sim \text{Powerlaw}(\tau,k_n,n)$, where $\tau=3$ and $k_n=\lfloor n^{2/3}\rfloor$. According to \cref{cor:recovery_power_law}, Algorithm~\ref{alg:commonNeighborsPartitioning} achieves both exact and almost exact recovery in this regime.
Figure~\ref{fig:almost-exact} reports the average correlation coefficient $\rho(T_n,\cD(G_n))$, together with confidence bounds. This quantity converges to $1$ as $n$ increases, indicating almost exact recovery. Figure~\ref{fig:exact} shows the fraction of trials in which exact recovery is achieved (i.e., $\cD(G_n)=T_n$); this fraction also approaches $1$ for large $n$. 
Because exact recovery is a much stronger requirement than almost exact recovery, it only becomes noticeable for larger~$n$: the event $\cD(G_n)=T_n$ only starts to be  observed for $n\ge4\cdot10^5,$ while $\rho(\cD(G_n),T_n)\approx0.99$ already holds for $n=10^5$.

\begin{figure}[!ht]
\centering
\begin{subfigure}[b]{0.48\linewidth}
    \includegraphics[width=\linewidth]{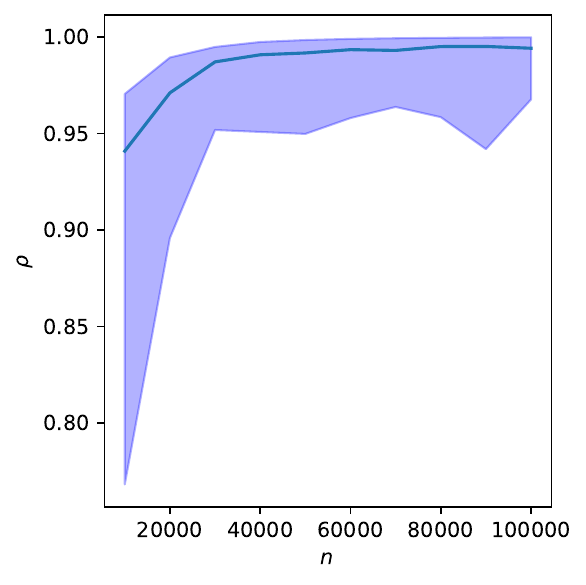}
    \caption{Correlation coefficient}
 \label{fig:almost-exact}
 \end{subfigure} \hfill
\begin{subfigure}[b]{0.48\linewidth}
    \includegraphics[width=\linewidth]{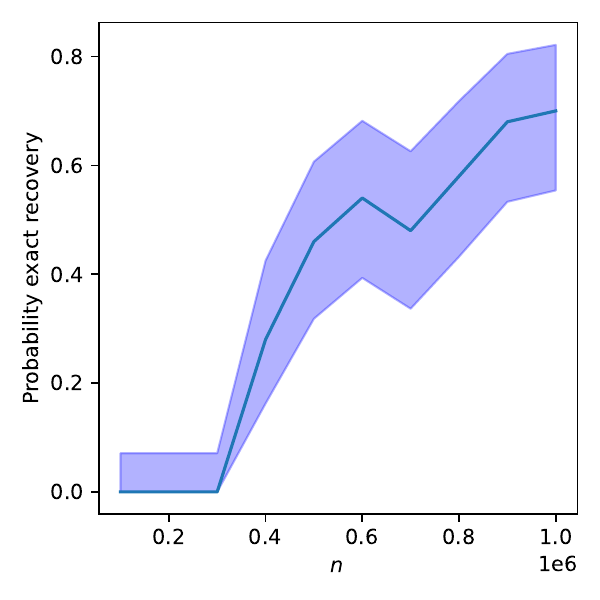}
    \caption{Fraction of exact recovery}
    \label{fig:exact}
\end{subfigure} 
    \caption{We apply $\cD$ to $G_n\sim\PPM(T_n,p,q_n)$ for $p=0.4,q_n=5/n$ and $T_n\sim\text{Powerlaw}(\tau,k_n,n)$ with power-law exponent $\tau=3$ and $k_n\sim n^{2/3}$ communities.\\ 
    Figure~\ref{fig:almost-exact}: For $n$ between $10^4$ and $10^5$, we report the average $\rho(T_n,\cD(G_n))$ among $120$ graph samples. The shaded areas denote $95\%$ confidence bounds. \\ 
    Figure~\ref{fig:exact}: For $n$ between $10^5$ and $10^6$, we report the fraction of times (among 50 runs) that $\cD$ achieved exact recovery (i.e., $\cD(G_n)=T_n$). We calculate $95\%$ confidence intervals using Wilson score intervals.}
    \label{fig:power-law-validation}
\end{figure}

 In Figure~\ref{fig:almost-exact-boltzmann}, we illustrate \cref{thm:almost_exact}, focusing in particular on the uniform-partition case described in \cref{ex:uniform}. We evaluate the performance of $\cD$ on PPMs in which the underlying partition is drawn uniformly at random from the set of all partitions of size $n$. Because sampling from the uniform distribution over partitions of $[n]$ is challenging, we employ Boltzmann samplers~\citep{duchon2004boltzmann} to generate a random partition $T_n$ whose number of vertices $N(T_n)$ is random but satisfies $\E[N(T_n)]=n$ and $N(T_n)/n \wprto 1$. A key property of this sampler is that, for any $n'$, conditional on $N(T_n)=n'$, the partition $T_n$ is uniformly distributed among all partitions of $[n']$. 

 Figure~\ref{fig:almost-exact-boltzmann} shows that the correlation coefficient converges very slowly to $1$. This behavior is expected, as $S_n=\log n+\bigO(\sqrt{\log n})$~\citep{gosgens2024erd}, which grows slowly with $n$. The figure also highlights a strong concentration of the correlation coefficient for fixed $n$, attributable to the large number of communities.  

\begin{figure}
    \centering
    \includegraphics[width=0.7\linewidth]{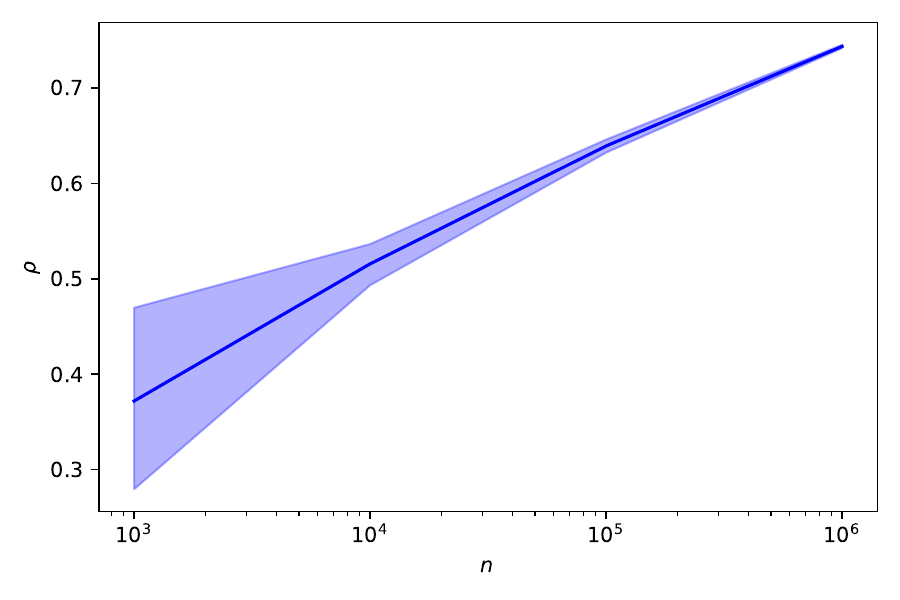}
    \caption{We apply $\cD$ to $G_n\sim\PPM(T_n,p,q_N)$ for $p=0.4,q_N=5/N$ and $T_n$ generated by a Boltzmann sampler. We report the average $\rho(T_n,\cD(G_n))$ among $40$ graph samples for each $n$. The shaded region denotes the maximum and minimum observed $\rho$.}
    \label{fig:almost-exact-boltzmann}
\end{figure}

Finally, we compare $\cD$ to two community detection methods that are widely used in practice. The first is the \emph{Louvain} algorithm~\citep{blondel2008fast}, which greedily maximizes \emph{modularity}~\citep{newman2006modularity}. The second is \emph{Bayesian stochastic block modeling}~\citep{peixoto2019bayesian}, which infers communities by maximizing the posterior probability within a Bayesian framework.

Figure~\ref{fig:comparison} reports the performance of these three methods on PPMs with a linear number of communities. Figure~\ref{fig:comparison-balanced} corresponds to balanced partitions, in which each community has size $20$, while Figure~\ref{fig:comparison-multinomial} considers multinomial partitions, where community sizes follow a binomial distribution with parameters $n$ and $20/n$. 
%Figure~\ref{fig:comparison} shows the performance of these three methods on PPMs with a linear number of communities. Figure~\ref{fig:comparison-balanced} shows the performance for balanced communities (each community has size $20$) while Figure~\ref{fig:comparison-multinomial} shows the performance for multinomial partitions (i.e., the size of each community is binomial with $n$ trials and success probability $20/n$). 
For small graphs, both the Louvain algorithm and the Bayesian method outperform $\cD$. As the graph size increases, however, the performance of $\cD$ stabilizes at a high value, whereas the performance of the competing methods degrades. These behaviors are consistent with known limitations of the Louvain and Bayesian approaches. In particular, the Louvain algorithm suffers from the \emph{resolution limit}~\citep{fortunato2007resolution}, which prevents it from reliably detecting small communities in large graphs. Similarly, Bayesian SBM methods are known~\citep{peixoto2017nonparametric} to struggle with communities of size $o(\sqrt n)$. In contrast, \cref{thm:weak} guarantees that $\cD$ can recover communities of bounded size. 
%We see that for small graphs, the Louvain algorithm and the Bayesian method outperform $\cD$. However, for larger graphs, the performance of $\cD$ stabilizes at a high value, while the performances of the other methods degrade. These are known defects of the Louvain algorithm and the Bayesian method: the Louvain algorithm is known to suffer from a \emph{resolution limit}~\citep{fortunato2007resolution}, which makes it unable to detect small communities in large graphs. Similarly, the Bayesian SBM is known~\citep{peixoto2017nonparametric} to struggle with communities of sizes $o(\sqrt n)$. In contrast, \cref{thm:weak} shows that $\cD$ is able to recover communities of bounded size.
Moreover, by comparing Figures~\ref{fig:comparison-balanced} and~\ref{fig:comparison-multinomial}, we observe that $\cD$ performs better on multinomial partitions than on balanced ones. This difference can be attributed to the heterogeneity in community sizes induced by the multinomial partitions and to the fact that $\cD$ more reliably recovers larger communities (see Figure~\ref{fig:scores-p}).
%Comparing Figure~\ref{fig:comparison-balanced} to Figure~\ref{fig:comparison-multinomial}, we see that $\cD$ performs slightly better on the multinomial partitions than the balanced partitions. This can be explained by the community-size heterogeneity of the multinomial partitions and fact that $\cD$ recovers larger communities better (see Figure~\ref{fig:scores-p}). 

\begin{figure}[!ht]
\centering
\begin{subfigure}[b]{0.48\linewidth}
    \includegraphics[width=\linewidth]{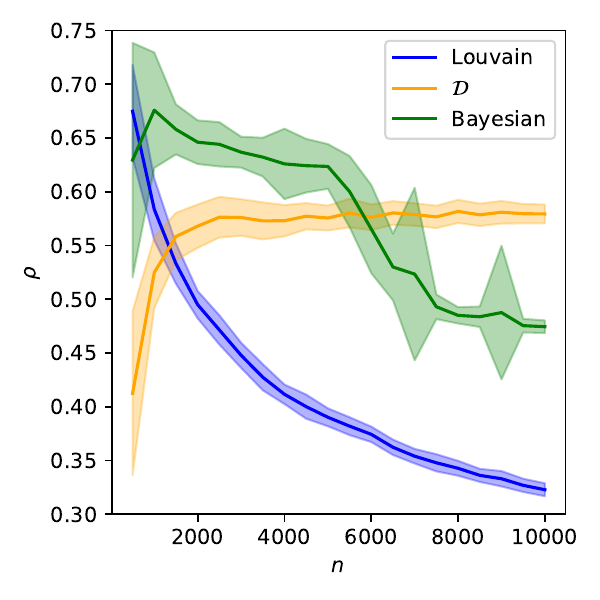}
    \caption{Balanced communities.}
 \label{fig:comparison-balanced}
 \end{subfigure} \hfill
\begin{subfigure}[b]{0.48\linewidth}
    \includegraphics[width=\linewidth]{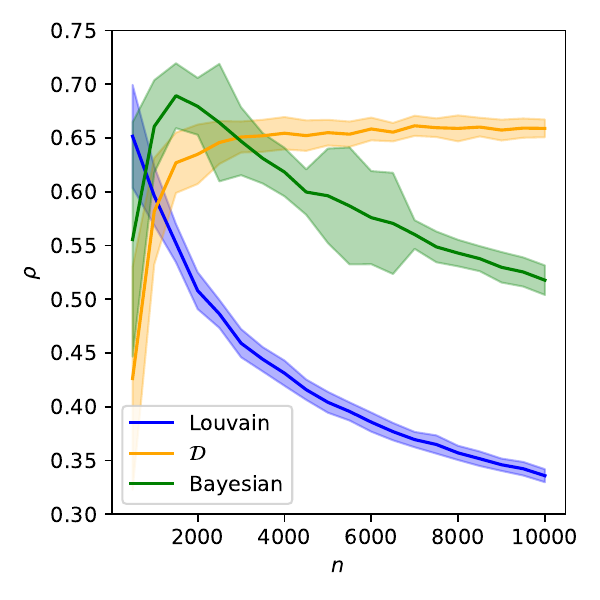}
    \caption{Multinomial partitions.}
    \label{fig:comparison-multinomial}
\end{subfigure} 
    \caption{We compare $\cD$ to the Louvain algorithm~\citep{blondel2008fast} and the Bayesian stochastic blockmodeling methodology~\citep{peixoto2019bayesian} for graphs of different size $n$ with $n/20$ communities, $p=5/19$ and $q_n=5/(n-20)$. We plot the average $\rho(C_n,T_n)$. The shaded regions denote one standard deviation.\\ Figure~\ref{fig:comparison-balanced}: comparison for PPMs with $T_n\sim\text{Balanced}(n,\lfloor n/20\rfloor,20)$; Figure~\ref{fig:comparison-multinomial}: comparison for PPMs with $T_n\sim\text{Multinomial}(n,[20/n,\dots,20/n])$.}
    \label{fig:comparison}
\end{figure}

\section{Discussion and Future Work}
\label{section:discussion}
In this work, we presented Diamond Percolation, a simple community detection method that runs in polynomial time and requires no parameter knowledge. We proved several conditions under which our algorithm achieves exact, almost exact, and weak recovery. In this section, we discuss ways in which the results could be extended and relate our methods to existing work.

\medskip
\textbf{Isolated vertices.} Because Diamond Percolation clusters vertices based on overlapping triangles, every vertex of degree less than three will be isolated in $G^*$. Therefore, such vertices will form singleton communities in $C_n$. Sometimes, as is the case in \cref{fig:example}, each of the edges of such a low-degree vertex connect to the same community. This suggests that we may be able to improve the performance by assigning such isolated low-degree vertices to a neighboring community. However, doing so may affect the validity of \cref{thm:refinement-generalized}. Improving Algorithm~\ref{alg:commonNeighborsPartitioning} by relabeling these low-degree vertices is an interesting direction for future research.

%\medskip

\medskip
\textbf{Size-dependent community densities.} Our results allow for communities of heterogeneous sizes, where the largest communities may differ orders of magnitude from the smallest communities. For a vertex in a community of size $s_n\gg1$, its degree is of the order $s_n p_n$. This leads to a linear dependence between the community size and the degree of a vertex. While it has been observed in practice that large-degree vertices tend to be part of large communities, this relation is typically sub-linear~\citep{stegehuis2016power}. This suggests that large communities should be sparser than small communities, which would be achieved by allowing for a size-dependent density $p_n(s_n)$ that decreases in~$s_n$ for fixed $n$.
We believe that most of our results can be extended to allow for these size-dependent intra-community densities, at the expense of more complicated expressions.

\medskip
 \textbf{Recovery thresholds.} We already discussed in Section~\ref{subsection:special_cases_PPM} some recent works establishing information-theoretic thresholds for weak and exact recovery in the PPM. However, unlike Diamond Percolation, the algorithms studied in these works typically assume knowledge of the model parameters and focus on balanced or near-balanced community structures. As a result, the corresponding information-theoretic thresholds fall outside the scope of the present article, which prioritizes parameter-free recovery under highly flexible community configurations. A systematic investigation of these thresholds under similarly minimal structural assumptions is an important avenue of future research. 

\medskip
\textbf{Recovering the largest communities.}
The difficulty of community detection is typically driven by the presence of small communities. As a result, recent works study the recovering of the \textit{largest} communities only (typically of size $\Omega(\sqrt{n})$) in the presence of an arbitrary number of smaller communities \citep{ailon2015iterative,mukherjee2024detecting}. While in this paper, we focus on the recovery of all communities, we could adapt Example~\ref{example:exact_recovery_balanced} in Section~\ref{subsection:exact_recovery} to the recovery of communities of size at least $s_n$ among smaller communities. %For example, Example~\ref{example:exact_recovery_balanced} to recover a community of size $s_n$ among many smaller communities. 

%\subsection{Algorithm~\ref{alg:commonNeighborsPartitioning} and Single-Linkage}
\medskip
\textbf{Single-linkage agglomerative clustering.}
The Diamond Percolation algorithm shares a conceptual similarity with single-linkage hierarchical clustering, a widely used agglomerative clustering method. In single-linkage clustering, two clusters are merged if they contain at least one pair of points that are sufficiently close, gradually forming larger clusters. Similarly, Diamond Percolation constructs a refined version of the input graph by preserving only edges that participate in at least two triangles, effectively filtering out weak connections. The final communities are then identified as the connected components of this filtered graph. This approach can be seen as a form of hierarchical clustering where the linkage criterion is based on common neighbors rather than direct pairwise distances. By setting a threshold of two shared neighbors, Diamond Percolation implicitly prioritizes denser local structures.%, making it particularly effective in detecting well-connected communities while avoiding the pitfalls of single-linkage clustering, such as chaining effects in sparse regions. 

Hence, one could study a version of Diamond Percolation with a threshold different than two. Varying the threshold would return a hierarchy of partitions, where larger thresholds give rise to finer partitions. Recent work studying hierarchical extension of the stochastic block model show that linkage algorithms recover the hierarchy of communities when the communities are of size $\Theta(n)$~\citep{dreveton2023does}. Our analysis hints that one could recover the hierarchy using linkage algorithms even when the communities are of much smaller size. 

\medskip
\textbf{Adjusting for degree heterogeneity.}
Furthermore, while we chose the number of shared triangles as the similarity measure between two vertices $i$ and~$j$, this metric can be modified based on their degrees. For example, \citet{michielan2022detecting} re-weigh triangles based on the degrees of the participating vertices to account for degree heterogeneity. Additionally, \citet{bonald2018hierarchical} proposes a distance measure based on node pair sampling. 

\medskip
\textbf{Geometric models.}
Finally, beside the stochastic block model, triangle-counting algorithms recover the planted partition in models with geometry~\citep{galhotra2023community}. In such models, the typical number of triangles is much larger than in the PPM, and communities can be detected based on shared neighborhoods better than in the PPM. Establishing the recovery of planted partitions in geometric models with small-size communities is an interesting avenue for future research.  

%%%%%%%%%%%%%%%%%%%%%%%%%%%%%%%%%%%%%%%%%%%%%%%%%%%%%%%%%%%%%%%%%%%
%%                                                               %%
%% Supplementary Material, if any, should be provided in         %%
%% {supplement} environment  with title and short description.   %%
%%                                                               %%
%%%%%%%%%%%%%%%%%%%%%%%%%%%%%%%%%%%%%%%%%%%%%%%%%%%%%%%%%%%%%%%%%%%

%%%%%%%%%%%%%%%%%%%%%%%%%%%%%%%%%%%%%%%%%%%%%%%%%%%%%%%%%%%%%%%%%%%
%%                                                               %%
%% Use the two commands below for producing your bibliography    %%
%% with bibtex, then comment again the commands and include the  %%
%% content of the .bbl file in this file below the commands.     %%
%%                                                               %%
%%%%%%%%%%%%%%%%%%%%%%%%%%%%%%%%%%%%%%%%%%%%%%%%%%%%%%%%%%%%%%%%%%%

\bibliographystyle{amsplain}
\bibliography{biblio}

% add below the content of your .bbl file produced by bibtex.

\appendix
\section{Proofs for Section~\ref{section:framework_algorithm}}

\subsection{Proof of Lemma~\ref{lemma:running_time_algo}}

\begin{proof}[Proof of Lemma~\ref{lemma:running_time_algo}]
 We compute $W_{ij}$ by computing the intersection between the neighborhoods of $i$ and $j$. If the neighborhoods are represented by hash sets, this intersection can be computed in $\bigO(\min\{d_i,d_j\})$. We bound $\min\{d_i,d_j\}\le d_i+d_j$. 
 For each edge adjacent to $i$, we have to compute an intersection, which results in a contribution of $d_i^2$. More explicitly, we write
    \[
    \sum_{ij\in E}d_i+d_j=\sum_{i\in[n]}d_i^2.
    \]
 The final step (computing the connected components of $E^*$) can be performed by a breadth-first search, which has time and space complexity $\bigO(n+|E|)$. We conclude that the time complexity is $\bigO(n+\sum_{i\in[n]}d_i^2)$.

 The space complexity follows from the fact that we only maintain the vertex neighborhoods and the edge sets $E$ and $E^*$. 
\end{proof}

%\subsection{Obtaining a refinement of an arbitrary partition}
\subsection{Proof of Theorem~\ref{thm:refinement-generalized}}
\label{appendix:proof_thm:refinement-generalized}

\begin{proof}[Proof of Theorem~\ref{thm:refinement-generalized}]
 Consider vertices $i,j$ with $i\nrel{T_n}j$ and let $s_i,s_j$ denote the sizes of their communities. We decompose $W_{ij}=W_{ij}^-+W_{ij}^+$, where $W_{ij}^-$ denotes the number of common neighbors outside the communities of $i$ and $j$, while $W_{ij}^+$ denotes the number of common neighbors that are in the union of $i$ and $j$'s communities. We write
    \begin{align*}
    \1 \{W_{ij}\ge2\}
    & \weq \1\{W_{ij}^+\ge2\}+\1\{W_{ij}^+=1\} \1\{W_{ij}^-\ge1\} + \1\{W_{ij}^+=0\} \1\{W_{ij}^-\ge2\} \\
    & \wle \1\{W_{ij}^+\ge2\} + \1\{W_{ij}^+=1\} \1\{W_{ij}^-\ge1\} + \1\{W_{ij}^-\ge2\}.
    \end{align*}
    When $i\nrel{T_n}j$, then $W_{ij}^+\sim \Binom(s_i+s_j-2,p_nq_n)$ and $W_{ij}^-\sim \Binom(n-s_i-s_j,q_n^2)$. Note that $W_{ij}^-$ is stochastically dominated by the random variable $X\sim\Binom(n,q_n^2)$. We obtain
    \begin{align}
        \cP{W_{ij}\ge2}{i\nrel{T_n}j}
        \wle &\ \cP{W_{ij}^+\ge2}{i\nrel{T_n}j}\nonumber\\
        &+\cP{W_{ij}^+=1}{i\nrel{T_n}j}\P(X\ge1)\nonumber\\
        &+\cP{W_{ij}^+=0}{i\nrel{T_n}j}\P(X\ge2)\nonumber\\
        \weq &\ \bigO\left((s_i+s_j)^2p_n^2q_n^2\right)+\bigO\left(n(s_i+s_j)p_nq_n^3\right)+\bigO(n^2q_n^4), \label{eq:general-bound}
    \end{align}
    where in the last step, we used that for $Y\sim\Binom(k,p)$, it holds that $\P(Y=0)\le 1$, $\P(Y=1)\le kp$ and $\P(Y\ge2)\le k^2p^2$.
    We bound $(s_i+s_j)^2 \le 2(s_i^2+s_j^2)$. 
    %Because $q_n = \bigO(n^{-1})$, the term on the right-hand side of~\eqref{eq:general-bound} is dominated by its first term. 
    %Hence, $\cP{W_{ij}\ge2}{i\nrel{T_n}j}=\bigO(p_n^2 q_n^2(s_i^2+s_j^2))$. 
    This tells us that there is some $c>0$ so that 
    \[
     \cP{W_{ij}\ge2}{i\nrel{T_n}j} \wle c \cdot (p_n^2 q_n^2 ( s_i^2 + s_j^2 )+n(s_i+s_j)p_nq_n^3+n^2q_n^4).
    \]
    We rewrite this as
    \begin{align*}
        &\cE{\1\{W_{ij}\ge2\}}{i\nrel{T_n}j,|T_n(i)|,|T_n(j)|}\\
        \wle& c \cdot \left(p_n^2 q_n^2 ( |T_n(i)|^2 + |T_n(j)|^2 )+n(|T_n(i)|+|T_n(j)|)p_nq_n^3+n^2q_n^4\right).
    \end{align*}
    Next, to get rid of the conditioning on $i\nrel{T_n}j$, we multiply with $\1\{i\nrel{T_n}j\}$ and write
    \begin{align}\label{eq:wedges-sizes-bound}
    &\cE{\1\{i\nrel{T_n}j\}\1\{W_{ij}\ge2\}}{|T_n(i)|,|T_n(j)|}\nonumber\\
        \wle& c \cdot\1\{i\nrel{T_n}j\}\cdot  \left(p_n^2 q_n^2 ( |T_n(i)|^2 + |T_n(j)|^2 )+n(|T_n(i)|+|T_n(j)|)p_nq_n^3+n^2q_n^4\right)\nonumber\\
        \wle& c \cdot  \left(p_n^2 q_n^2 ( |T_n(i)|^2 + |T_n(j)|^2 )+n(|T_n(i)|+|T_n(j)|)p_nq_n^3+n^2q_n^4\right).
    \end{align}
    To prove the theorem, we use $\P(C_n\preceq T_n)=1-\P(\exists_{ij}:i\rel{G_n^*}j\wedge i\nrel{T_n}j)$ and bound, by Markov's inequality, 
    \begin{align}\label{eq:refinement-Markov}
        \P(\exists_{ij}:i\rel{G_n^*}j\wedge i\nrel{T_n}j)
        & \wle \E\left[ \#\{ij\ :\ i\rel{G_n^*}j\wedge i\nrel{T_n}j\} \right] \nonumber\\
        & \weq \sum_{ij} \P(i\nrel{T_n}j) \cP{i\rel{G_n^*}j}{i\nrel{T_n}j} \nonumber\\
        & \wle \sum_{i\in[n],j\in [n]} \P(i\nrel{T_n}j) \cP{i\rel{G_n^*}j}{i\nrel{T_n}j}.
    \end{align}
    We rewrite 
    \begin{align}\label{eq:diamond-rewrite}
        \P(i\nrel{T_n}j)\cP{i\rel{G_n^*}j}{i\nrel{T_n}j}
        & \weq q_n \P(i\nrel{T_n}j) \cP{W_{ij}\ge2}{i\nrel{T_n}j} \\
        & \weq q_n \E\left[\cE{\1\{i\nrel{T_n}j\}\1\{W_{ij}\ge2\}}{|T_n(i)|,|T_n(j)|}\right].
    \end{align}
    
    where we used \eqref{eq:wedges-sizes-bound} and the tower rule.
    Substituting \eqref{eq:diamond-rewrite} into \eqref{eq:refinement-Markov}, and applying the bound from \eqref{eq:wedges-sizes-bound} leads to
    \begin{align}\label{eq:refinement-final-bound}
        \P(\exists_{ij}:i\rel{G_n^*}j\wedge i\nrel{T_n}j)\le c\cdot n^2\cdot q_n\left(p_n^2q_n^2\cdot 2\E[S_n^2]+n\cdot2\E[S_n]p_nq_n^3+n^2q_n^4\right).
    \end{align}
    We need this quantity to vanish. For the last term to vanish, we require $n^4q_n^5\to0$, i.e., $q_n\ll n^{-4/5}$. For the first term to vanish, we need
    \[
        n^2q_n^3p_n^2\E[S_n^2]\to0.
    \]
    If these two terms vanish, then the other term must also vanish. To see this, note that the geometric mean of $n^4q_n^5$ and $n^2q_n^3p_n^2\E[S_n^2]$ is upper-bounds the middle term by Jensen's inequality:
    \[
        \sqrt{n^2q_n^3p_n^2\E[S_n^2]\cdot n^4q_n^5}=n^3q_n^4p_n\sqrt{\E[S_n^2]}\ge n^3q_n^4p_n\E[S_n].
    \]
    Hence, if the first and third term of \eqref{eq:refinement-final-bound} vanish, it implies that $\P(\exists_{ij}:i\rel{G_n^*}j\wedge i\nrel{T_n}j)\to0$, so that
    \[
    \P(C_n\preceq T_n)\to1,
    \]
    which completes the proof.
\end{proof}

\subsection{Proof of Lemma~\ref{lem:correlation}}

\begin{proof}[Proof of Lemma~\ref{lem:correlation}]
Recall the definition of the correlation coefficient between two partitions given in~\eqref{eq:correlation}. 
Whenever $C\preceq T$, we have $m_{CT}=m_C\le m_T$, so that \eqref{eq:correlation} becomes
\begin{align*}
    \rho(C,T)
    &=\frac{m_C\cdot (N-m_T)}{\sqrt{m_C\cdot(N-m_C)\cdot m_T\cdot(N-m_T)}}\\
    &=\sqrt{\frac{m_C}{m_T}\cdot \frac{N-m_T}{N-m_C}}.
\end{align*}
    By \cref{assumption:concentration_mT}, $\E[m_T]\ll N$. Markov's inequality tells us that $\P(m_T>\sqrt{N\E[m_T]})\le \sqrt{\E[m_T]/N}\to0$.
    Conditioned on the events $C\preceq T$ and $m_T\le \sqrt{N\E[m_T]}\ll N$ and calculate
    \begin{align*}
        \rho(C,T)
        &=\frac{N\cdot m_{CT}-m_Cm_T}{\sqrt{m_C\cdot(N-m_C)\cdot m_T\cdot (N-m_T)}}\\
        &=\frac{m_C\cdot(N-m_C)}{\sqrt{m_C\cdot(N-m_C)\cdot m_T\cdot (N-m_T)}}\\
        &=\sqrt{\frac{m_C}{m_T}}\cdot \sqrt{\frac{N-m_C}{N-m_T}}\sim\sqrt{\frac{m_C}{m_T}},
    \end{align*}
    where the last step follows from the fact that $m_C\le m_T\ll N$. Therefore, for any $\varepsilon>0$, it holds that
    \[
    \cP{\left|\rho(C,T)-\sqrt{\frac{m_C}{m_T}} \right|>\varepsilon}{C\preceq T,m_T\le \sqrt{N\E[m_T]}}=o(1).
    \]
    Since the events that we condition on occur with high probability, we conclude
    \begin{align*}
    & \P\left(\left|\rho(C,T)-\sqrt{\frac{m_C}{m_T}} \right|>\varepsilon\right)\\
    & =\cP{\left|\rho(C,T)-\sqrt{\frac{m_C}{m_T}} \right|>\varepsilon}{C\preceq T,m_T\le \sqrt{N\E[m_T]}}(1-o(1)) + o(1)\rightarrow 0.
    \end{align*}
\end{proof}

\section{Proofs for General Partitions (Section~\ref{section:arbitrary_partitions})}
\label{appendix:proof-arbitrary-partitions}

\subsection{Proofs of Theorem~\ref{thm:exact}}

\begin{proof}
[Proof of Theorem~\ref{thm:exact}]
    We prove that w.h.p., all intra-community vertex pairs have at least two common neighbors inside their community, so that
    all edges in the planted community are contained in $G^*$. This additionally implies that the community is connected and has diameter at most $2$. \cref{thm:refinement-generalized} additionally guarantees that none of the edges outside the planted community are contained in $G^*$. Together, this guarantees exact recovery.

    We prove that w.h.p., there is no pair $ij$ of vertices in the community with $W_{ij}^+<2$.
    Conditioned on $T$, we use the Markov inequality to bound the probability that there exists such a pair by $\sum_{i\rel{T}j}\P(W_{ij}^+<2)$.
    Suppose the community of $i$ and $j$ has size $s$. Then, the subgraph induced by their community corresponds to an ER graph with $s$ vertices and connection probability $p_n$.
    Since $W_{ij}^+\sim\text{Bin}(s-2,p_n^2)$, we can write
    \begin{align*}
    \P(W_{ij}^+<2)
    &=(1-p_n^2)^{s-2}+(s-2)p_n^2(1-p_n^2)^{s-3}\\
    &=(1-p_n^2+(s-2)p_n^2)(1-p_n^2)^{s-3}\\
    &=(1+(s-3)p_n^2)(1-p_n^2)^{s-3}\\
    &\le (1+sp_n^2)(1-p_n^2)^{s-3}.
    \end{align*}
    Taking the derivative of its log, we see that this quantity is decreasing for
    \begin{align*}
    &\frac{p_n^2}{1+sp_n^2}+\log(1-p_n^2)\le0\\
    \Rightarrow &s\ge \frac{1}{-\log(1-p_n^2)}-p_n^{-2}=\frac{p_n^2+\log(1-p_n^2)}{-p_n^2\log(1-p_n^2)}=\frac{1}{2}+\bigO(p_n^2).
    \end{align*}
    This allows us to upper-bound this probability by substituting $s=s_n^{(\min)}$.
    This yields
    \begin{equation}\label{eq:exact-final-bound}
    \P(\exists i,j:\ i\rel{T}j\wedge W_{ij}^+<2) \wle \E[m_T]\cdot(1+s_n^{(\min)}p_n^2)(1-p_n^2)^{s_n^{(\min)}-3}.
    \end{equation}
    We show that the above requires $s_n^{(\min)}p_n^2\to\infty$. 
    Suppose $s_n^{(\min)}p_n^2=\bigO(1)$, then
    \[
    (1-p_n^2)^{s_n^{(\min)}}= e^{-p_n^2s_n^{(\min)}+\bigO(p_n^4s_n^{(\min)})}=\Omega(1),
    \]
    since $p_n^4s_n^{(\min)}=\left(p_n^2s_n^{(\min)}\right)^2/s_n^{(\min)}\to0$.
    This implies that the right-hand-side of~\eqref{eq:exact-final-bound} is $\Omega(\E[m_T])\to\infty$, which indeed forms a contradiction.
    Now, using $s_n^{(\min)}p_n^2\to\infty$, we rewrite the right-hand-side of~\eqref{eq:exact-final-bound} to
    \[
    \E[m_T]\cdot(1+s_n^{(\min)}p_n^2)(1-p_n^2)^{s_n^{(\min)}-3}=\bigO\left(\E[m_T]\cdot s_n^{(\min)}p_n^2e^{-p_n^2s_n^{(\min)}}\right).
    \]
    This tells us that we need
    \[
    \E[m_T]\cdot p_n^2s_n^{(\min)}e^{-p_n^2s_n^{(\min)}}=o(1).
    \]
    If we take $p_n$ as given in the theorem, then
    \[
    p_n^2s_n^{(\min)}=\log\E[m_T]+\log\log \E[m_T]+\omega,
    \]
    for some sequence $\omega\to\infty$. Substituting this yields
    \begin{align*}
        \E[m_T]\cdot p_n^2s_n^{(\min)}e^{-p_n^2s_n^{(\min)}}
        &=\frac{\E[m_T]\cdot \left(\log\E[m_T]+\log\log \E[m_T]+\omega\right)}{\E[m_T]\cdot\log\E[m_T]\cdot e^\omega}\\
        &=e^{-\omega}\cdot\left(1+\frac{\log\log\E[m_T]+\omega}{\log\E[m_T]}\right)\to0,
    \end{align*}
    as required.
\end{proof}

\subsection{Proof of Theorem~\ref{thm:almost_exact}}

\begin{proof}[Proof of Theorem~\ref{thm:almost_exact}]
 By \cref{thm:refinement-generalized}, we have $\P(C_n\preceq T_n)\rightarrow 1$ . Moreover, \cref{lem:correlation} implies $\rho(C_n,T_n)^2-\frac{m_{C_n}}{\E[m_{T_n}]}\stackrel{\P}{\rightarrow}0$. 
    We will prove that $\tfrac{\E[m_C]}{\E[m_T]}\rightarrow1$, so that $\rho(C_n,T_n)^2\stackrel{\P}{\rightarrow}1$. 
    Let us define the random variable $S_n'=|C_n(I_n)|$. Noticing that $m_C=\frac{1}{2}\sum_{i\in[n]}(|C(i)|-1)$, we have 
    \[
    \frac{\E[m_C]}{\E[m_T]} \weq \frac{\frac{n}{2}\E[S_n'-1]}{\frac{n}{2}\E[S_n-1]} \weq \frac{\E[S_n']-1}{\E[S_n]-1}.
    \]
    We have
    \begin{equation}\label{eq:S_n-biased}
    \frac{S_n'-1}{\E[S_n-1]} \wge \frac{(S_n-1)\cdot\mathbb1\{S_n=S_n'\}}{\E[S_n-1]}.
    \end{equation}
    We introduce the function $\phi(s)$ given by
    \[
    \phi(s)=\cP{S_n=S_n'}{S_n=s}.
    \]
    Moreover, we denote by $S_n^*$ the distribution of $S_n$ biased by $S_n-1$. That is,
    \[
    \P(S_n^*=s)=\frac{(s-1)\P(S_n=s)}{\E[S_n-1]}.
    \]
    This allows us to write the expectation of the right-hand-side of~\eqref{eq:S_n-biased} as
    \[
    \E\left[\frac{(S_n-1)\cdot\mathbb1\{S_n=S_n'\}}{\E[S_n-1]}\right] \weq \E[\phi(S_n^*)]
    \]
    From \cref{assumption:minimum_size_soft}, $S_n^*\ge s_n^{(\min)}$ holds with high probability.
    The remainder of the proof is similar to the proof of \cref{thm:exact}: if $S_n=s_n$, then the probability that there exists a vertex pair in the community of $I_n$ that does not have two internal common neighbors is 
    \[
    1-\phi(s_n) \wle \cP{\exists ij\in{T_n(I_n)\choose 2} \colon W_{ij}^+<2}{S_n=s_n} \weq \bigO(s_n^3p_n^2e^{-s_np_n^2}),
    \]
    where 
    ${T_n(I_n)\choose 2}$ denotes the set of vertex pairs belonging to $T_n(I_n) \times T_n(I_n)$. 
    Because the quantity inside the $\bigO$ in the previous expression is decreasing in $s_n$, we can use the bound $S_n\ge s_n^{(\min)}$ to write
    \[
    \phi(s_n^{(\min)}) \weq 1-\bigO((s_n^{(\min)})^3p_n^2e^{-s_n^{(\min)}p_n^2}).
    \]
    We need this error term to vanish. 
    Using a similar derivation as in the proof of~\cref{thm:exact}, it is sufficient to have
    \begin{align*}
    s_n^{(\min)}p_n^2& \weq \log((s_n^{(\min)})^2)+\log\log((s_n^{(\min)})^2)+\omega(1)\\
    & \weq 2\log s_n^{(\min)}+\log\log s_n^{(\min)}+\omega(1),
    \end{align*}
    which can indeed be rearranged to the expression presented in the theorem.
    We conclude that
    \[
    \E[\rho(C_n,T_n)^2] \wge \E[\phi(S_n^*)]\to1,
    \]
    so that $\rho(C_n,T_n)\stackrel\P\to1$ as required.
\end{proof}

\subsection{Proof of Lemma~\ref{lem:uniform-integrability}}
\label{subsection:proof_lemma_uniform-integrability}
\begin{proof}[Proof of~\cref{lem:uniform-integrability}] 
    First, we prove that $\mathbb E[S_n]$ is bounded. If not, then we can take $s_n=\frac12\max_{n'\le n}\mathbb E[S_{n'}]\to\infty$ and take the subsequence $n_k$ for which $\mathbb E[S_{n}]=\max_{n'\le n}\mathbb E[S_{n'}]$, and write
$$
\varepsilon_{n_k}(s_{n_k})=
\frac{\mathbb E[(S_{n_k}-1)\mathbb 1\{S_{n_k}< s_{n_k}\}]}{\mathbb E[S_{n_k}-1]}<\frac{s_{n_k}-1}{\E[S_{n_k}-1]}\to\frac12,
$$
which contradicts the claim that $\varepsilon_n(s_n)\to1$. We conclude that $\mathbb E[S_n]$ is bounded.
The lemma's assumption is equivalent to
$$
1-\varepsilon_n(s_n)=\frac{\mathbb E[(S_n-1)\mathbb 1\{S_n\ge s_n\}]}{\mathbb E[S_n-1]}\to0.
$$
A sequence of random variables is uniformly integrable if for every $\varepsilon>0$, there exists some $s<\infty$ so that for all $n$,
\[
\E\left[S_n\cdot\mathbb 1\{S_n\ge s\}\right]\le\varepsilon.
\]
Again, we prove this claim via contradiction. Suppose there is some $\varepsilon>0$, so that for all $s<\infty$, there is some $n(s,\varepsilon)$ for which $\E\left[S_n\cdot\mathbb 1\{S_n\ge s\}\right]>\varepsilon$. Then, for any sequence $s_k\to\infty$, we can construct the subsequence $n_k=n(s_k,\varepsilon)$, along which
\[
1-\varepsilon_{n_k}(s_k)=\frac{\mathbb E[S_{n_k}\cdot\mathbb 1\{S_{n_k}\ge s_k\}]}{\mathbb E[S_{n_k}]}>\frac{\varepsilon}{\mathbb E[S_{n_k}]}\ge \frac{\varepsilon}{\sup_n\mathbb E[S_{n}]}>0,
\]
again contradicting the assumption $\varepsilon_{n_k}(s_k)\to1$, which proves that $S_n$ must indeed be uniformly integrable. 

Uniform integrability implies that $S_n$ is a tight sequence of random variables. By Prokhorov's theorem, every tight sequence of random variables has a convergent subsequence $S_{n_\ell}$. Uniform integrability implies that the mean of this subsequence must also converge to the mean of its limit, which completes the proof.
\end{proof}

\subsection{Proof of Theorem~\ref{thm:weak}}

\subsubsection{Two Technical Lemmas}
Let $G'_n$ denote the graph of intra-community edges of $G_n$. That is $i\rel{G'}j$ if and only if $i\rel{G}j$ and $i\rel{T}j$. Then we define $C_n'=\cD(G_n')$ as the partition that results from applying our algorithm to $G'_n$. Now, since every edge of $G_n'$ is contained in $G_n$, every edge of $(G_n')^*$ is also contained in $G^*_n$. This implies that  $C_n'\preceq C_n$. We denote the number of intra-cluster pairs of $C_n'$ by $m_{C'}$.

It would obviously be much easier to recover $T_n$ from $G_n'$ than from $G_n$, since every edge of $G_n'$ is guaranteed to connect two vertices of the same community. Because of this, we would expect $\cD$ to perform better on $G_n'$ than $G_n$. Counterintuitively, the next lemma proves that the opposite is true: with high probability, the $C_n'$ provides a \emph{lower bound} on the performance of $C_n$. 

\begin{lemma}\label{lem:further-refinement}
    If \cref{assumption:size-sparsity} holds, then
    \[
    \rho(C_n,T_n) \wge \rho(C_n',T_n) 
    \]
    with high probability. If additionally \cref{assumption:concentration_mT} holds, then also
    \[
    \rho(C_n',T_n)-\sqrt{\frac{m_{C'}}{\E[m_T]}}\wprto0.
    \]
\end{lemma}
\begin{proof}
    Since every edge of $G_n'$ is present in $G_n$, every edge of $(G_n')^*$ is also present in $G_n^*$, which implies $C_n'\preceq C_n$. Hence, $C_n'\preceq C_n\preceq T_n$ holds with high probability.
    In \cite{gosgens2021systematic}, it was proven that the correlation coefficient $\rho$ is monotone with respect to merging communities. That is, $C_n'\preceq C_n\preceq T_n$ implies $\rho(C_n,T_n)\ge\rho(C_n',T_n)$.

    Finally, note that $G_n' \sim \PPM(T_n,p,0)$ (i.e., by setting $q_n=0$). Therefore, applying \cref{lem:correlation} to $G'_n$ yields the last claim.
\end{proof}

The fact that we can lower-bound the performance of $C_n$ by the performance of $C_n'$ is convenient, since $C_n'$ is much easier to analyze.

\begin{lemma}\label{lem:mC-concentration}
    If Assumptions~\ref{assumption:concentration_mT} and~\ref{assumption:convergent-s} hold, then
    \[
        \frac{m_{C'}}{\E[m_T]}-\frac{\E[m_{C'}]}{\E[m_T]}\wprto0.
    \]
\end{lemma}
\begin{proof}
    We write the number of intra-community pairs in the $a$-th community of $T_n$ as
    \[
        m_T^{(a)}={|T_n^{(a)}|\choose 2},
    \]
    so that $m_T=\sum_{a}m_T^{(a)}$.
    We divide $T_n$ into small and large communities, where we set the threshold at
    \[
    s_n=\E[m_T]^{1/3}.
    \]
    We define the set of small communities as
    \[
    A_<=\{a\ :\ |T_n^{(a)}|<s_n\},
    \]
    and define $A_{\ge}$ similarly as the set of communities of size at least $s_n$.
    Let
    \[
    m_T^<=\sum_{a\in A_<} m_T^{(a)},
    \]
    denote the sum of intra-community pairs in these smaller communities and define let $m_T^\ge=m_T-m_T^<$.
    We write
    \begin{align*}
        \E\left[m_T^<\right]
        & \weq \E\left[\frac12\sum_{i\in [n]} (|T_n(i)|-1)\cdot \1\{|T_n(i)|< s_n\} \right]\\
        & \weq \frac{n}{2}\cE{(S_n-1)\cdot\1\{S_n< s_n\}}{S_n>1}\cdot \P(S_n>1)\\
        & \wsim \frac n2\E[S-1]\cdot \P(S_n>1), 
    \end{align*}
    where the last line follows from $((S_n-1)\cdot\1\{S_n<s_n\}\ |\ S_n>1)\stackrel{d}{\to}S-1$ because $s_n\to\infty$ (Assumption~\ref{assumption:convergent-s}).
    Similarly,
    \begin{align*}
        \E[m_T]& \weq \frac n2\cE{S_n-1}{S_n>1}\cdot\P(S_n>1)\\
        & \wsim \frac n2\E[S-1]\cdot \P(S_n>1),
    \end{align*}
    so that $\E[m_T]\sim \E[m_T^<]$. Therefore,
    \[
    \frac{\E[m_T^\ge]}{\E[m_T]} \weq \frac{\E[m_T]-\E[m_T^<]}{\E[m_T]}\to0.
    \]
    By Markov's inequality, it follows that
    \begin{equation}\label{eq:large-communities-Markov}
    \frac{m_T^\ge}{\E[m_T]}\wprto0.
    \end{equation}

    We now define the number of \emph{recovered} intra-community pairs of the $a$-th community of $T_n$ as
    \[
    m_{C'}^{(a)} \weq \#\left\{(i,j)\in {T_n^{(a)}\choose 2}\ :\ i\rel{C'}j\right\}.
    \]
    From this definition, we have $0\le m_{C'}^{(a)}\le m_T^{(a)}$ and $m_{C'} \weq \sum_am_{C'}^{(a)}$. Define also 
    \[
    m_{C'}^< \weq \sum_{a\in A_<}m_{C'}^{(a)},
    \]
    and $m_{C'}^\ge$. Firstly, we can bound $m_{C'}^\ge\le m_T^\ge$ and use \eqref{eq:large-communities-Markov} to write
    \[
    \frac{m_{C'}^\ge}{\E[m_T]}\wprto0.
    \]

    Secondly, the random variables $(m_{C'}^{(a)})_{a\in A_<}$ are independent when conditioned on the true partition $T_n$, because $m_{C'}^{(a)}$ only depends on the edges inside $T_n^{(a)}$.
    This allows us to apply Hoeffding's inequality to the conditional probability
    \begin{align*}
        \cP{|m_{C'}^<-\E[m_{C'}^<]|>\varepsilon\E[m_T]}{T_n}
        & \wle 2\exp \left( -\frac{\varepsilon^2\E[m_T]^2}{\sum_{a\in A_<} (m_T^{(a)})^2} \right),
    \end{align*}
    where we used $0\le m_{C'}^{(a)}\le m_T^{(a)}$ for each $a\in A_<$.
    In this bound, the quantity $\sum_{a\in A_<} (m_T^{(a)})^2$ is a function of the random variable~$T_n$.
    Note that the function $x\mapsto e^{-1/x}$ is increasing for $x>0$. We can bound
    \begin{eqnarray*}
    \sum_{a\in A_<} (m_T^{(a)})^2
    =&\sum_{a\in A_<} {|T_n^{(a)}|\choose 2}^2
    &\le s_n^2\sum_{a\in A_<} {|T_n^{(a)}|\choose 2}\\
    \le& s_n^2m_T
    &=m_T\E[m_T]^{2/3}.
    \end{eqnarray*}
    This allows us to bound
    \begin{align}
        \P(|m_{C'}^<-\E[m_{C'}^<]|>\varepsilon\E[m_T])
        & \weq \E\left[\cP{|m_{C'}^<-\E[m_{C'}^<]|>\varepsilon\E[m_T]}{T_n}\right]\nonumber\\
        & \wle 2 \, \E\left[ \exp\left( -\frac{\varepsilon^2\E[m_T]^2}{\sum_{a\in A_<} (m_T^{(a)})^2}\right) \right] \nonumber\\
        & \wle 2 \, \E\left[ \exp\left( -\varepsilon^2\E[m_T]^{1/3}\frac{\E[m_T]}{m_T}\right) \right]\nonumber\\
        &\to0,
    \end{align}
    where we used $m_T/\E[m_T]\wprto1$ and $\E[m_T]^{1/3}\to\infty$ from \cref{assumption:concentration_mT}. We have shown
    \[
        \frac{m_{C'}^<-\E[m_{C'}^<]}{\E[m_T]}\wprto0.
    \]
    Putting everything together, the conclude that
    \begin{align*}
        \frac{m_{C'}}{\E[m_T]}-\frac{\E[m_{C'}]}{\E[m_T]}
        & \weq \frac{m_{C'}^<-\E[m_{C'}^<]}{\E[m_T]}+\frac{m_{C'}^\ge }{\E[m_T]}-\frac{\E[m_{C'}^\ge]}{\E[m_T]}\\
        & \weq o_\P(1)+o_\P(1)-o(1).
    \end{align*}
\end{proof}

\subsubsection{Proof of Theorem~\ref{thm:weak}}

\begin{proof}[Proof of \cref{thm:weak}]

\begin{comment}
  By~\cref{lem:further-refinement} and \cref{assumption:size-sparsity} imply that $\rho(C_n,T_n)\ge \rho(C_n',T_n)$ w.h.p., and
    \[
    \rho(C'_n,T_n)-\sqrt{\frac{m_{C'}}{m_T}}\wprto 0.
    \]
    Moreover, because $\frac{m_T}{\E[m_T]}\stackrel{\P}{\rightarrow}1$ (\cref{assumption:concentration_mT}), we have 
    \[
    \rho(C_n',T_n)-\sqrt{\frac{m_{C'}}{\E[m_T]}}\wprto 0.
    \]
\end{comment}
    By~\cref{lem:further-refinement}, we have 
    \[
    \rho(C_n',T_n)-\sqrt{\frac{m_{C'}}{\E[m_T]}}\wprto 0.
    \]
    Moreover, \cref{lem:mC-concentration} implies
    \[
    \rho(C_n',T_n)-\sqrt{\frac{\E[m_{C'}]}{\E[m_T]}}\wprto 0.
    \]
    Note that $\E[m_T]=\frac{n}{2}\E[S_n-1]=\frac n2\cE{S_n-1}{S_n>1}\cdot\P(S_n>1)$.
    Thus, we have 
    \begin{align}
        \frac{\E[m_{C'}]}{\E[m_T]}
        \weq & \frac{\sum_{i\in[n]}\cE{|C'_n(i)|-1}{S_n>1}\P(|T_n(i)|>1)}{n\cE{S_n-1}{S_n>1}\P(S_n>1)} \nonumber \\
        \weq & \frac{\cE{|C'_n(I_n)|-1}{S_n>1}}{\cE{S_n-1}{S_n>1}}, \label{eq:in_proof_fraction}
    \end{align}
    where the randomness is taken over $I_n$, which is a vertex uniformly distributed over $[n]$, and $S_n=|T_n(I_n)|$.
    
    By assumption, the denominator of~\eqref{eq:in_proof_fraction} converges to $\E[S-1]$. 
    To compute the numerator $\cE{|C'_n(I_n)|-1}{S_n>1}$, note that $|C'_n(I_n)|$ depends only on the edges in the community of $I_n$. The subgraph of $G$ induced by $T_n(I_n)$ is equal in distribution to an \ER random graph $H_n$ with $|T_n(I_n)|=S_n$ vertices and connection probability $p$, i.e., $H_n\sim\text{ER}(S_n,p)$.
    This leads to a coupling between $|C_n(I_n)|$ and $|C^{(H_n)}(1)|$, where $C^{(H_n)}=\cD(H_n)$. In this coupling, we replaced the arbitrary vertex $I_n$ by the first vertex of $H_n$. 
    Because $\cE{|C^{(H_n)}(1)|}{S_n}\le \E[S_n|S_n>1]\to\E[S]$ and $(S_n\ |\ S_n>1)\xrightarrow{d}S$, the dominated convergence theorem implies 
    \[
    \E[|C^{(H_n)}(1)|-1]\rightarrow \E[|C^{(H)}(1)|-1],
    \]
    for $H\sim\text{ER}(S,p)$. 
\end{proof}

\section{Proofs for Power-law Partitions (Section~\ref{section:power_law_partitions})}

\subsection{Size-bias Distribution}
The following two lemmas establish the convergence in distribution of $k \Pi_a$ and of $k \Pi^*$, respectively.

\begin{lemma}
 \label{lem:powerlaw-tail}
 For $\alpha\in [\tfrac1k,1)$, the tail probability of $\Pi_a$ defined in~\eqref{eq:def_Pi_a} is given by
 \begin{equation}\label{eq:pi-exact}
    \P(\Pi_a>\alpha) \weq \left(\frac{1-\alpha}{(k-1)\alpha}\right)^\tau\cdot \E\left[ \left(\frac{1}{k-1}\sum_{b\neq a}e^{X_b/\tau}\right)^{-\tau} \right].
\end{equation}
Moreover, when $k \gg 1$, we have 
 \[
  k\Pi_a \stackrel d \wto \text{Pareto} \left( 1-\tfrac1\tau,\tau \right).
 \]
\end{lemma}

\begin{proof}[Proof of Lemma~\ref{lem:powerlaw-tail}]
We rewrite
\begin{align}
 \P(\Pi_a>\alpha)
 & \weq \P\left(\frac{e^{X_a/\tau}}{e^{X_a/\tau}+\sum_{b\neq a}e^{X_b/\tau}}>\alpha\right) \nonumber \\
 & \weq \P\left(e^{X_a/\tau}>\frac{\alpha}{1-\alpha}\sum_{b\neq a}e^{X_b/\tau}\right) \nonumber \\
 & \weq \P\left(X_a>\tau\log\left(\frac{\alpha}{1-\alpha}\right)+\tau\log\left(\sum_{b\neq a}e^{X_b/\tau}\right)\right) \nonumber \\
 & \weq \P\left(X_a > \tau\log\left(\frac{(k-1)\alpha}{1-\alpha}\right) + \tau\log\left(\frac{1}{k-1}\sum_{b\neq a}e^{X_b/\tau}\right)\right) \label{eq:in_proof_expo}. 
 %&= \left(\frac{1-\alpha}{(k-1)\alpha}\right)^\tau\cdot \E\left[ \left(\frac{1}{k-1}\sum_{b\neq a}e^{X_b/\tau}\right)^{-\tau} \right],
\end{align}
Note that $e^{X/\tau}\ge1$, so that the second log term is nonnegative. However, $\log\left(\frac{(k-1)\alpha}{1-\alpha}\right)$ is negative for $\alpha<\tfrac1k$. Hence, for $\alpha\ge\tfrac1k$, we have
 \[
 \P(\Pi_a>\alpha) \weq \left(\frac{1-\alpha}{(k-1)\alpha}\right)^\tau\cdot \E\left[ \left(\frac{1}{k-1}\sum_{b\neq a}e^{X_b/\tau}\right)^{-\tau} \right],
 \]
which proves~\eqref{eq:pi-exact}. For $\alpha<\tfrac1k$, we condition on the value of $Y_k=\tfrac{1}{k-1}\sum_{b\neq a}e^{X_b/\tau}$ and write
\begin{align*}
 \cP{\Pi_a>\alpha}{Y_k}
 & \weq \min\left\{1,\left(\frac{1-\alpha}{(k-1)\alpha}\right)^\tau\cdot Y_k^{-\tau} \right\}.
\end{align*}
Taking the expectation with respect to $Y_k$, we obtain
\begin{align}
 \P(\Pi_a>\alpha)
 & \weq \P\left(Y_k<\frac{1-\alpha}{(k-1)\alpha}\right)+\E\left[\left(\frac{1-\alpha}{Y_k(k-1)\alpha}\right)^\tau\cdot\1\left\{Y_k\ge\frac{1-\alpha}{(k-1)\alpha}\right\}\right]\nonumber\\
 & \weq \left(\frac{1-\alpha}{(k-1)\alpha}\right)^\tau\cdot\E[Y_k^{-\tau}]+\E\left[\left(1-\left(\frac{1-\alpha}{Y_k(k-1)\alpha}\right)^\tau\right)\cdot\1\left\{Y_k<\frac{1-\alpha}{(k-1)\alpha}\right\}\right]\label{eq:tail-annoying}.
\end{align}
By the weak law of large numbers, we have
\[
 Y_k \weq \frac{1}{k-1}\sum_{b\neq a}e^{X_b/\tau} \wprto \E\left[ e^{X_b/\tau} \right] \weq \frac{1}{1- \frac{1}{\tau} }.
\]
In addition, $f(x)=x^{-\tau}$ is bounded for $x\ge1$, so that the above convergence implies
\[
 \E\left[ Y_k^{-\tau} \right] \wto \left(1-\frac1\tau\right)^\tau.
\]
After substituting $\alpha=z/k$ for $z\ge1,$ this already tells us that
\[
 \left(\frac{1-z/k}{\frac{k-1}{k}z}\right)^\tau\cdot\E[Y_k^{-\tau}] \wto \left(\frac{1-\tfrac1\tau}{z}\right)^\tau.
\]
To complete the proof, we need to show that for $1-\tfrac1\tau<z\le 1$, the second term in~\eqref{eq:tail-annoying} vanishes. Note that $1-\tfrac1\tau<z\le 1$ implies
\[
 \frac{1-z/k}{(k-1)z/k} \weq z^{-1}+\frac{z^{-1}-1}{k-1} \wto z^{-1}<(1-\tfrac1\tau)^{-1},
\]
so that
\[
 \1\left\{Y_k<\frac{1-z/k}{(k-1)z/k}\right\} \wprto 0,
\]
because $Y_k\wprto(1-\tfrac1\tau)^{-1}>z^{-1}$. This tells us that the second term in~\eqref{eq:tail-annoying} vanishes for $\alpha=z/k$, which completes the proof.
\end{proof}

\begin{lemma}\label{lem:pi-size-bias}
    Let $\Pi^*$ be the proportion of the community of a random vertex~$I_n$. Then
    \[
    k\Pi^*\stackrel d\to\text{Pareto}\left( 1-\tfrac1\tau,\tau-1 \right).
    \]
    Moreover, for $r>\tau-1$ and $\beta\in[0,1]$, we have 
    \[
    \E[(\Pi^*)^r\cdot\1\{\Pi^*<k^{\beta-1}\}] \weq \bigO(k^{(r+1-\tau)\beta-r}),
    \]
    while for $r<\tau-1$, we have 
    \[
    \E[(\Pi^*)^r] \wsim k^{-r}\frac{(\tau-1)^{1+r}}{\tau^r(\tau-1-r)}.
    \]
\end{lemma}

\begin{proof}[Proof of Lemma~\ref{lem:pi-size-bias}]
    For $\alpha>1/k$, we rewrite~\eqref{eq:pi-exact} to
    \[
    \P(\Pi_a>\alpha) \weq c_{k,\tau}\cdot(\tfrac1\alpha-1)^\tau,
    \]
    where
    \[
    c_{k,\tau} 
    \weq \E\left[ \left(\sum_{b\neq a}e^{X_b/\tau}\right)^{-\tau} \right] 
    \wsim \left(\frac{1-\tfrac1\tau}{k}\right)^{\tau}.
    \]
    Taking the derivative w.r.t. $\alpha$ yields the density
    \[
    f(\alpha) 
    \weq -\frac{d}{d\alpha}\P(\Pi_a>\alpha)
    \weq \frac{c_{k,\tau}\tau}{\alpha^2} \left( \frac1\alpha - 1 \right)^{\tau-1}
    \weq \frac{c_{k,\tau}\tau(1-\alpha)^{\tau-1}}{\alpha^{\tau+1}}.
    \]
    To derive the density of $\Pi_a$ for $(1-\tfrac1\tau)/k<\alpha\le 1/k$, we need to take the second term in~\eqref{eq:tail-annoying} into account. We rewrite it to
    \[
    \P(\Pi_a>\alpha)=c_{k,\tau}\cdot(\tfrac1\alpha-1)^\tau+\int_1^{\frac{1-\alpha}{(k-1)\alpha}}f_{Y_k}(y)\cdot\left(1-\left(\frac{1-\alpha}{y(k-1)\alpha}\right)^\tau\right)dy,
    \]
    where $f_{Y_k}(y)$ is the density of
    \[
    Y_k=\frac1{k-1}\sum_{b\neq a}e^{X_b/\tau}.
    \]
    Taking the derivative of the second term w.r.t. $\alpha$, we obtain
    \begin{align*}
        &\frac d{d\alpha}\int_1^{\frac{1-\alpha}{(k-1)\alpha}}f_{Y_k}(y)\cdot\left(1-\left(\frac{1-\alpha}{y(k-1)\alpha}\right)^\tau\right)dy\\
        =&\left(\frac d{d\alpha}\frac{1-\alpha}{(k-1)\alpha}\right)\cdot f_{Y_k}\left(\frac{1-\alpha}{(k-1)\alpha}\right)\cdot(1-1)+\int_1^{\frac{1-\alpha}{(k-1)\alpha}} f_{Y_k}(y)\tau\frac{\left(\tfrac1\alpha-1\right)^{\tau-1}\alpha^{-2}}{y^\tau(k-1)^\tau}dy\\
        =&\frac{\tau(\tfrac1\alpha-1)^{\tau-1}}{(k-1)^\tau\alpha^2}\E\left[Y_k^{-\tau}\cdot\1\left\{Y_k<\frac{1-\alpha}{(k-1)\alpha}\right\}\right]\\
        =&o\left(k^{-\tau}\alpha^{-\tau-1}\right),
    \end{align*}
    where the last step follows from
    \[
    \frac{1-\alpha}{(k-1)\alpha}=\frac{1}{k\alpha}+\frac{\tfrac1{k\alpha}-1}{k-1}>\tfrac{1}{k\alpha},
    \]
    and 
    \[
    \E\left[Y_k^{-\tau}\cdot\1\left\{Y_k<\frac{1-\alpha}{(k-1)\alpha}\right\}\right]\le\E\left[Y_k^{-\tau}\cdot\1\left\{Y_k<\frac{1}{k\alpha}\right\}\right]\to0,
    \]
    by the weak law of large numbers.
    This tells us that for $(1-\tfrac1\tau)/k<\alpha\le 1/k$, the density of $\Pi_a$ is
    \[
    f(\alpha)=\frac{c_{k,\tau}\tau(1-\alpha)^{\tau-1}}{\alpha^{\tau+1}}+o\left(k^{-\tau}\alpha^{-\tau-1}\right).
    \]
    Since $\Pi^*$ is the size-biased distribution of $\Pi_a$, its density is given by
    \begin{align*}
    f^*(\alpha)
    &\weq \frac{\alpha f(\alpha)}{\E[\Pi_a]}\\
    &\weq k\alpha f(\alpha)\\
    &\weq \frac{kc_{k,\tau}\tau(1-\alpha)^{\tau-1}}{\alpha^{\tau}}+o\left(k^{1-\tau}\alpha^{-\tau}\right).
    \end{align*}
    The density of $k\Pi^*$ is then
    \[
    g_k(z) 
    \weq k^{-1}f^*(z/k)
    \weq \frac{k^\tau c_{k,\tau}\tau(1-z/k)^{\tau-1}}{z^{\tau}}+o\left(z^{-\tau}\right),
    \]
    for $z\ge 1-\tfrac1\tau$. This converges point-wise to
    \[
    g(z) \weq \tau \left( 1- \tfrac{1}{\tau} \right)^\tau z^{-\tau}.
    \]
    Then, by Scheffé's theorem~\citep{scheffe1947useful}, $k\Pi^*$ converges in distribution to a random variable with density $g(z)$. Integrating $g(z)$ tells us that the corresponding tail function should be 
    \[
    \P(k\Pi^*>z)\to\int_z^\infty g(y)dy
    \weq \frac{\tau}{\tau-1} \left( 1-\tfrac1\tau \right)^\tau z^{1-\tau}
    \weq \left(\frac{1-\tfrac1\tau}{z}\right)^{\tau-1},
    \]
    which indeed corresponds to $\text{Pareto}\left(1-\tfrac1\tau,\tau-1\right)$.

    We now prove the asymptotics of the moments. For $r<\tau-1$, we use the dominated convergence theorem. 
    Let $c^*=\sup_kk^\tau c_{k,\tau}\tau$, which is finite since this sequence converges.
    We bound
    \begin{equation}\label{eq:powerlaw-biased-density-bound}
    g_k(z) 
    \wle \frac{k^\tau c_{k,\tau}\tau}{z^{\tau}}
    \wle c^*z^{-\tau},
    \end{equation}
    so that
    \[
    \E[(k\Pi^*)^r] \wle c^*\int_{1-\tfrac1\tau}^\infty z^{r-\tau}dz,
    \]
    since $r-\tau<-1$.
    The dominated convergence theorem then allows us to interchange the limit and integration, so that
    \begin{align*}
    \E[(k\Pi^*)^r]
    &\wto \int_{1-1/\tau}^\infty z^rg(z)dz\\
    & \weq \tau(1-\tfrac{1}{\tau})^\tau\int_{1-1/\tau}^\infty  z^{r-\tau}dz\\
    & \weq \frac{\tau}{\tau-1-r}(1-\tfrac{1}{\tau})^\tau\cdot(1-\tau^{-1})^{1+r-\tau}\\
    & \weq \frac{(\tau-1)^{1+r}}{\tau^r(\tau-1-r)}. 
    \end{align*}
    For $r>\tau-1$, we similarly write
    \begin{align*}
        \E[(k\Pi^*)^r\cdot\1\{k\Pi^*<k^\beta\}] 
        \wle c^*\int_{1-\tfrac1\tau}^{k^\beta}z^{r-\tau}dz
        \weq \bigO\left( k^{\beta(r+1-\tau)}\right),
    \end{align*}
    so that indeed $\E[(\Pi^*)^r\cdot\1\{\Pi^*<k^{\beta-1}\}] \weq \bigO(k^{\beta(r+1-\tau)-r})$.
    
    %For the final asymptotics, we return to the density of $\Pi^*$ and write
    %\[
    %    \E[(\Pi^*)^r]=\int_{\alpha_-}^1\alpha^rf^*(\alpha)d\alpha\le \int_{0}^1\alpha^rf^*(\alpha)d\alpha,
    %\]
    %where $\alpha_-\ge0$ is the left limit of the support of $\Pi^*$. We relate this integral to the Beta function:
    %\[
    %\int_{0}^1\alpha^rf^*(\alpha)d\alpha=kc_{k,\tau}\tau \int_0^1\alpha^{r-\tau}(1-\alpha)^{\tau-1}d\alpha=kc_{k,\tau}\tau B(r+1-\tau,\tau),
    %\]
    %where
    %\[
    %B(z_1,z_2)=\int_0^1t^{z_1-1}(1-t)^{z_2-1}dt.
    %\]
    %Since $B(r+1-\tau,\tau)<\infty$ and $kc_{k,\tau}=\bigO(k^{1-\tau})$, it follows that $\E[(\Pi^*)^r]=\bigO(k^{1-r})$.
\end{proof}

\subsection{Proof of Theorem~\ref{thm:power-law-rescaled}}
\begin{proof}[Proof of Theorem~\ref{thm:power-law-rescaled}]
    We prove that
    \[
    \frac{k_n S_n}{n} - k_n \Pi^* \wprto 0,
    \]
    so that the result follows from \cref{lem:pi-size-bias}.
    Using Chebyshev's bound, it suffices to show that
    \[
     \E\left[ k_n^2 \left( \frac{S_n}{n}-\Pi^* \right)^2 \right] \weq o(1).
    \]
    Conditioned on $\Pi^*$, $S_n-1$ is binomially distributed with $n-1$ trials and success probability $\Pi^*$. The variance is
    \[
    \cE{ \left( S_n-1-(n-1)\Pi^* \right)^2}{\Pi^*} \weq (n-1)\Pi^*(1-\Pi^*)
    \]
    Multiplying by $k_n^2/n^2$, we obtain
    \begin{equation}
    \label{eq:power-law-bound}
    \cE{ k_n^2 \left( \frac{S_n}{n} - \Pi^* - \frac{1-\Pi^*}{n} \right)^2 }{\Pi^*} \weq \frac{n-1}{n^2} k_n \Pi^* (k_n-k_n\Pi^*) 
    \wle \frac{k_n^2}{n} \Pi^*.
    \end{equation}
    We rewrite the left-hand-side as
    \begin{align*}
    & k_n^2\left( \frac{S_n}{n} - \Pi^* - \frac{1-\Pi^*}{n} \right)^2 \\
    & \weq k_n^2 \left( \frac{S_n}{n} - \Pi^* \right)^2 + \frac{k_n^2(1-\Pi^*)^2}{n^2}-\frac{2k_n^2(1-\Pi^*)}{n} \left( \frac{S_n}{n}-\Pi^* \right).
    \end{align*}
    We take the conditional expectation given $\Pi^*$.
    Using $\cE{S_n}{\Pi^*}=n\Pi^*+1-\Pi^*$, we obtain
    \[
    \cE{k_n^2 \left( \frac{S_n}{n} -\Pi^* \right)^2}{\Pi^*}-\frac{k_n^2(1-\Pi^*)^2}{n^2}.
    \]
    Because $k_n \ll n$, the last term vanishes. Taking the expectation of \eqref{eq:power-law-bound} w.r.t. $\Pi^*$ on both sides, we conclude
    \[
    \E\left[k_n^2 \left( \frac{S_n}{n} -\Pi^* \right)^2 \right] + o(1) \wle \frac{k_n}{n} \E\left[ k_n \Pi^* \right] \weq o\left( \frac{k_n}{n} \right).
    \]
\end{proof}

\begin{proof}[Proof of \cref{lem:powerlaw-mixed-Poisson}]
    Similarly as in the proof of \cref{thm:power-law-rescaled}, we condition on~$\Pi^*$ and use the fact that $S_n-1\sim\text{Bin}(n-1,\Pi^*)$ and consider the Poisson approximation with parameter $(n-1)\Pi^*$. We use \citet{roos2001sharp} to bound the difference between the probability mass function of the binomial and the Poisson distribution:
    \[
    \sup_{r=0,\dots,\infty}|\P(\text{Bin}(n,p)=r)-\P(\text{Poi}(np)=r)|\le np^2,
    \]
    so that
    \[
    -(n-1)(\Pi^*)^2\le\cP{S_n-1=r}{\Pi^*}-\frac{((n-1)\Pi^*)^r}{r!}e^{-(n-1)\Pi^*}\le (n-1)(\Pi^*)^2.
    \]
    We now take the expectation w.r.t. $\Pi^*$. \cref{lem:pi-size-bias} tells us that $\E[(\Pi^*)^2]=\bigO(k^{1-\tau})=\bigO(n^{1-\tau})$, so that
    \[
    \P(S_n=r+1)-\E\left[\frac{((n-1)\Pi^*)^r}{r!}e^{-(n-1)\Pi^*}\right]=\bigO(n^{2-\tau})\to0.
    \]
    Using \cref{lem:pi-size-bias}, $(n-1)\Pi^*\stackrel{d}\to \text{Pareto}(s\cdot (1-\tfrac1\tau),\tau-1)$ since $n-1\sim s\cdot k_n$. Since $x\mapsto x^se^{-x}$ is a bounded function, we conclude that 
    \[
    \P(S_n=r+1)=\E\left[\frac{((n-1)\Pi^*)^r}{r!}e^{-(n-1)\Pi^*}\right]+o(1)\to \E\left[\frac{Z^r}{r!}e^{-Z}\right].
    \]

\end{proof}

\subsection{Minimum Community Size in a Power-law Partition}
\label{subsection:powerlaw-minsize}

\begin{lemma}\label{lem:powerlaw-minsize}
 Let $\varepsilon>0$ and suppose $T_n\sim\text{Powerlaw}(\tau,k_n,n)$ for $\tau>2$ and $1\ll k_n \le \tfrac{\varepsilon^2}{4}\tfrac{\tau-1}{\tau}\tfrac{n}{\log n}$. Then with high probability, all communities are larger than $(1-\varepsilon)\frac{\tau-1}{\tau}\frac{n}{k_n}$. That is,
    \[
    \P\left(\exists i\in[n]:\ |T_n(i)|\le(1-\varepsilon)\frac{\tau-1}{\tau}\frac{n}{k_n} \right)\to0.
    \]
\end{lemma}
\begin{proof}[Proof of \cref{lem:powerlaw-minsize}]
    We first study the distribution of $\min_{a\in[k]}\{\Pi_a\}$.
    Note that $X_a<X_b$ implies $\Pi_a<\Pi_b$.
    Hence, the $a$ that minimizes $\Pi_a$ is the one with the minimal $X_a$. The minimum among these $k_n$ exponentially distributed random variables is exponentially distributed with rate $k_n$. Given that $a^*$ is the minimizer, the distribution of $X_b-X_{a^*}$ is exponential with rate $1$ for $b\neq a^*$. This allows us to write
    \begin{align*}
    \Pi_{a^*} & \weq \frac{e^{X_{a^*}/\tau}}{\sum_{a\in[k]}e^{X_a/\tau}}
    \weq \left(1+\sum_{a\neq a^*} e^{(X_a-X_{a^*})/\tau}\right)^{-1}.
    \end{align*}
    Then, by the weak law of large numbers, 
    \[
    k_n \Pi_{a^*} \weq \frac{k_n}{1+\sum_{a\neq a^*} e^{(X_a-X_{a^*})/\tau}} \wprto 1-\frac{1}{\tau}.
    \]

    Given $\Pi_a$, the distribution of the $a$-th community is binomially distributed with $n$ trials and success probability $\Pi_a$.
    The Markov inequality allows us to upper-bound the probability that there is a community smaller than $(1-\varepsilon)\frac{\tau-1}{\tau}\frac{n}{k_n}$ by the expected number of such small communities.
   This yields
    \[
    \P\left(\exists i\in[n]:\ |T_n(i)|\le(1-\varepsilon)\frac{\tau-1}{\tau}\frac{n}{k_n} \right)
    \wle \E\left[\sum_{a\in[k_n]}\P\left(\text{Bin}(n,\Pi_a)
    \wle (1-\varepsilon)\frac{\tau-1}{\tau}\frac{n}{k_n}\right)\right].
    \]
    Since $\text{Bin}(n,\Pi_a)$ stochastically dominates $\text{Bin}(n,\Pi_{a^*})$, the above is upper-bounded by
    \[
     k_n \E\left[\P\left(\text{Bin}(n,\Pi_{a^*}) \wle (1-\varepsilon) \frac{\tau-1}{\tau}\frac{n}{k_n} \right)\right].
    \]
    Since $x\mapsto\P\left(\text{Bin}(n,x)\le (1-\varepsilon)\frac{\tau-1}{\tau}\frac{n}{k_n}\right)$ is a bounded function, the weak law of large numbers tells us that 
    \[
    \E\left[\P\left(\text{Bin}(n,\Pi_{a^*}) \le (1-\varepsilon)\frac{\tau-1}{\tau}\frac{n}{k_n}\right)\right] 
    \wto \P\left(\text{Bin} \left( n, \frac{\tau-1}{k_n\tau} \right) \le (1-\varepsilon)\frac{\tau-1}{\tau}\frac{n}{k_n}\right).
    \]
    The Chernoff bound tells us that for $x,\varepsilon\in(0,1)$ it holds that
    \[
    \P(\text{Bin}(n,x) \le n\cdot (1-\varepsilon)x) 
    \wle e^{-n\cdot d_{KL}((1-\varepsilon)x\|x)},
    \]
    where $d_{KL}(y\|x)$ is the Kullback-Leibler divergence and can be lower-bounded by
    \[
    d_{KL}((1-\varepsilon)x\|x) \wge \frac{\varepsilon^2}{4}x,
    \]
    for $\varepsilon\in(0,\tfrac{1}{2})$.
    Taking these together, we obtain
    \[
    \P\left(\text{Bin}\left(n,\frac{\tau-1}{k_n\tau} \right) \le (1-\varepsilon) \frac{\tau-1}{\tau}\frac{n}{k_n}\right)
    \wle \exp\left( -\frac{\varepsilon^2}{4} \frac{\tau-1}{\tau} \frac{n}{k_n} \right).
    \]
    By our assumption on $k_n$, 
    \[
    \exp\left( -\frac{\varepsilon^2}{4} \frac{\tau-1}{\tau} \frac{n}{k_n} \right) \wle \frac{1}{n}.
    \]
    We conclude that
    \[
    \P\left(\exists i\in[n]:\ |T_n(i)|\le(1-\varepsilon)\frac{\tau-1}{\tau}\frac{n}{k_n} \right)
    \wle \frac{k_n}{n}
    \weq o\left(\frac{1}{\log n}\right).
    \]    
\end{proof}

\subsection{Additional Lemmas}
\label{subsection:additional_lemmas_power-law}

\begin{lemma}\label{lem:powerlaw-condition}
    Let $\tau>2$ and $\max\{\sqrt{n},n^{\frac{1}{\tau-1}}\}\ll k_n\le n$. If $T_n\sim\text{Powerlaw}(\tau,k_n,n)$, then $\E[S_n^2]=o(n)$.
\end{lemma}

\begin{proof}
 Conditioned on the value of $\Pi^*$, $S_n-1\sim\text{Bin}(n-1,\Pi^*)$. Hence, 
 \begin{align*}
    \cE{S_n^2}{\Pi^*}
    & \weq \cE{1+2(S_n-1)+(S_n-1)^2}{\Pi^*}\\
    & \weq 1+2(n-1)\Pi^*+(n-1)\Pi^*(1-\Pi^*)+(n-1)^2(\Pi^*)^2\\
    & \weq 1+3(n-1)\Pi^*+(n-1)(n-2)(\Pi^*)^2.
 \end{align*}
 Taking the expectation w.r.t. $\Pi^*$, we obtain
 \[
  \E \left[ S_n^2 \right] \weq 1+3(n-1)\E[\Pi^*]+(n-1)(n-2)\E[(\Pi^*)^2].
 \]
 \cref{lem:pi-size-bias} tells us that for $\tau>3$, $\E[\Pi^*]=\bigO(k_n^{-1})$ and $\E[(\Pi^*)^2]=\bigO(k_n^{-2})$. So that we need $k_n \gg \sqrt{n}$ to ensure $n^2 k_n^{-2} = o(n)$. For $\tau\in(2,3]$, $\E[(\Pi^*)^2]=\bigO(k_n^{1-\tau})$, so that we require $k_n \gg n^{\tfrac{1}{\tau-1}}$ to ensure $n^2 k_n^{1-\tau}=o(n)$.
\end{proof}

\begin{lemma}\label{lem:powerlaw-mT}
    If $T_n\sim\text{Powerlaw}(\tau,k_n,n)$ for $\tau>2$ and $k_n\to\infty$, then
    \[
    \frac{m_{T_n} }{\E[m_{T_n}]} \wprto 1,
    \]
    and
    \[
    \E[m_{T_n}] \wsim \frac{n^2(\tau-1)^2}{2k_n\tau(\tau-2)}.
    \]
\end{lemma}
\begin{proof}
 To avoid overburdening the notations, we write $T$ instead of $T_n$. First, we compute the expectation as
 \[
  \E[m_T] \weq \sum_{i<j}\P(i\rel{T}j).
 \]
 Now, we write
    \[
    \P(i\rel{T}j) 
    \weq \E\left[\sum_{a\in[k]}\Pi_a^2\right] 
    \weq k_n \E[\Pi_a^2].
    \]
    Because $\Pi^*$ is the size-biased version of $\Pi_a$, their moments are related by
    \[
    \E[(\Pi^*)^r] 
    \weq \frac{\E[\Pi_a^{1+r}]}{\E[\Pi_a]}
    \weq k_n \E[\Pi_a^{1+r}],
    \]
    so that $k_n \E[\Pi_a^2] = \E[\Pi^*]$. Using \cref{lem:pi-size-bias}, we obtain
    \[
    \E[m_T]
    \weq {n\choose 2}\E[\Pi^*] 
    \wsim \frac{n^2}{2}\frac{(\tau-1)^2}{k_n\tau(\tau-2)}.
    \]
    To show that $m_T/\E[m_T]\wprto1$, we distinguish two cases.

\medskip
    \textbf{The case $\tau>3$.}
    We write
    \[
    (m_T)^2
    \weq \left(\sum_{i<j}\1(i\rel{T}j)\right)^2.
    \]
    We distinguish the different products of indicators based on the number of distinct vertices that are involved. There are ${n\choose 2}$ terms that involve two vertices (products of indicators with itself), ${n\choose 2}\cdot{n-2\choose 2}$ terms with four distinct vertices, and ${n\choose 2}\cdot 2(n-2)$ terms that involve three distinct vertices.
    By symmetry, this allows us to write
    \begin{align}
    \E[m_T^2] \weq 
    & {n\choose 2}\P(1\rel{T}2)\nonumber\\
    & + {n\choose 2}\cdot {n-2\choose 2}\P(1\rel{T}2\wedge3\rel{T}4)\nonumber\\
    & + {n\choose 2}2(n-2)\P(1\rel{T}2\rel{T}3).
    \label{eq:powerlaw-mT-second-moment}
    \end{align}
    To show convergence, we need to show $\E[m_T^2]\sim \E[m_T]^2$. The first term of \eqref{eq:powerlaw-mT-second-moment} is $\E[m_T]=o(\E[m_T]^2)$.
    The third term of \eqref{eq:powerlaw-mT-second-moment} can be computed using \cref{lem:pi-size-bias}. We have
    \begin{align*}
    \P(1\rel{T}2\rel{T}3) 
    \weq k_n \E \left[ \Pi_a^3 \right]
    \weq \E \left[ (\Pi^*)^2 \right]
    \weq \bigO \left( k_n^{-2} \right),
    \end{align*}
    so that this term is also negligible.

    The term $\P(1\rel{T}2\wedge3\rel{T}4)$ requires some extra steps. We first write
    \begin{align}
    \P(1\rel{T}2\wedge3\rel{T}4) 
    \weq \P(1\rel{T}2\rel{T}3\rel{T}4) + \P(1\rel{T}2\nrel{T}3\rel{T}4). 
    \label{eq:in_proof_4way}
    \end{align}
    The first term of~\eqref{eq:in_proof_4way} is equal to $\E[(\Pi^*)^3] = \bigO(k_n^{1-\tau}) = o(k_n^{-2})$.
    To handle the second term of~\eqref{eq:in_proof_4way}, we sum over all possible labels for the community containing vertices $\{1,2\}$ and vertices $\{3,4\}$. We write
    \[
    \P(1\rel{T}2\nrel{T}3\rel{T}4) 
    \weq \sum_{a\neq b}\E \left[ \Pi_a^2\Pi_b^2 \right]
    \weq k_n(k_n-1) \E \left[ \Pi_1^2\Pi_2^2 \right].
    \]
    Using the definitions of $\Pi_1,\Pi_2$, this can be rewritten to
    \[
    \E \left[ \Pi_1^2\Pi_2^2 \right] 
    \weq \E\left[\frac{e^{2X_1/\tau}e^{2X_2/\tau}}{\left(\sum_{a\in[k_n]}e^{X_a/\tau}\right)^4}\right] 
    \wsim \frac{(\tau-1)^4}{k_n^4\tau^2(\tau-2)^2},
    \]
    where we used the strong law of large numbers and $\E[e^{tX_1}]=(1-t)^{-1}$. It follows that
    \[
    \P(1\rel{T}2\wedge3\rel{T}4) \wsim \frac{(\tau-1)^4}{k_n^2\tau^2(\tau-2)^2}.
    \]
    Putting these together, we obtain that for $\tau>3$,
    \[
    \E[m_T^2] \wsim \frac{n^4}{4}\frac{(\tau-1)^4}{k_n^2\tau^2(\tau-2)^2} \wsim \E[m_T]^2.
    \]
    This implies that $\Var(m_T/\E[m_T]) = o(1)$, so that $m_T/\E[m_T]\wprto1$ for $\tau>3.$

    \medskip
    \textbf{The case $2<\tau\le3.$}
    We define
    \[
    m_T^L \weq \sum_{a\in[k_n]}{|T_a|\choose 2}\cdot\1\left\{\Pi_a<k_n^{-\tfrac12} \right\}.
    \] 
    We use the Markov inequality to show that $m_T=m_T^L$ holds with high probability for $\tau\in(2,3]$:
    \begin{align*}
        \P(m_T^L\neq m_T)
        & \weq \P \left(\exists a\in[k]:\ \Pi_a\ge k_n^{-\tfrac12}\right)\\
        & \wle k_n \P \left( \Pi_a\ge k_n^{-\tfrac12} \right) \\
        & \weq k_n \cdot \bigO\left( \left( k_n\cdot k_n^{-\tfrac12} \right)^{-\tau}\right)\\
        & \weq \bigO \left( k_n^{1-\tfrac\tau2} \right) \wto 0,
    \end{align*}
    where we used \cref{lem:powerlaw-tail} and $\tau>2$. 
    Additionally, we show that $\E[m_T^L]\sim\E[m_T]$, or equivalently, that $\E[m_T-m_T^L]=o(n^2/k_n)$. We write
    \begin{align*}
    \E \left[ m_T-m_T^L \right]
    & \weq {n\choose 2} \E \left[ \Pi^*\cdot \1 \left\{ \Pi^* > k_n^{-\tfrac12} \right\} \right],
    \end{align*}
    Using the upper bound on the density of $k_n \Pi^*$ from \eqref{eq:powerlaw-biased-density-bound}, we obtain
    \begin{align*}
        k_n^{-1}\E[k_n \Pi^*\cdot\1\{k_n\Pi^*>\sqrt{k_n}\}]
        & \wle \frac{c^*}{k_n} \int_{\sqrt{k_n}}^\infty z^{1-\tau}dz
        \weq \bigO(k_n^{(2-\tau)/2-1}) 
        \weq o(k_n^{-1}),
    \end{align*}
    so that indeed $\E[m_T^L]\sim\E[m_T]$.

    In the remainder of the proof, we use Chebyshev's inequality to prove that $m_T^L/\E[m_T^L]\wprto1$. Let $\Pi^*(i)$ denote the $\Pi_a$ that corresponds to the community that the vertex $i$ is assigned to.
    Similarly to \eqref{eq:powerlaw-mT-second-moment}, we write the second moment of $m_T^L$ as a sum of indicators, and we distinguish the different products as follows 
    \begin{align}
    \E[(m_T^L)^2] \weq & {n\choose 2} \P \left( 1\rel{T}2,\Pi^*(1)<k_n^{-\tfrac12} \right) \nonumber\\
    &+{n\choose 2}\cdot {n-2\choose 2}\P \left( 1\rel{T}2\wedge3\rel{T}4,\Pi^*(1)<k_n^{-\tfrac12},\Pi^*(3)<k_n^{-\tfrac12} \right) \nonumber\\
    &+{n\choose 2}2(n-2) \P \left( 1\rel{T}2\rel{T}3,\Pi^*(1)<k_n^{-\tfrac12} \right).
    \label{eq:powerlaw-mTL-second-moment}
    \end{align}
    The first term of \eqref{eq:powerlaw-mTL-second-moment} is $\E[m_T^L]=o(\E[m_T^L]^2)$. For the third term in~\eqref{eq:powerlaw-mTL-second-moment}, we use \cref{lem:pi-size-bias} to compute
    \begin{align*}
    \P\left( 1\rel{T}2\rel{T}3,\Pi^*(1) < k_n^{-\tfrac12} \right)
    & \weq k_n \E\left[ \Pi_a^3\cdot\1\{\Pi_a < k_n^{-\tfrac12}\} \right] \\
    & \weq \E \left[ (\Pi^*)^2\cdot\1\{\Pi^* < k_n^{-\tfrac12}\} \right] \\
    & \weq \bigO\left( k_n^{(3-\tau)/2-2} \right),
    \end{align*}
    so that the third term of \eqref{eq:powerlaw-mTL-second-moment} contributes as $\bigO(n^3k^{(3-\tau)/2-2})=o(n^4k^{-2})$.
    
    The second term of \eqref{eq:powerlaw-mTL-second-moment} requires extra work. Firstly,
    \begin{align*}
    \P(1\rel{T}2\rel{T}3\rel{T}4,\Pi^*(1)<k^{-\tfrac12})
    & \weq \E\left[(\Pi^*)^3\cdot\1\{\Pi^*<k^{-\tfrac12}\}\right]\\
    & \weq \bigO\left(k_n^{(3+1-\tau)\frac{1}{2}-3}\right) 
    \weq \bigO(k_n^{-\frac\tau2-1}) \weq o(k_n^{-2}),
    \end{align*}
    where we used \cref{lem:pi-size-bias} in the last step. Secondly,
    \begin{align*}
    &\P(1\rel{T}2\nrel{T}3\rel{T}4,\Pi^*(1) < k_n^{-\tfrac12},\Pi^*(3) < k_n^{-\tfrac12}) \\
    \weq & k_n(k_n-1) \E[\Pi_1^2\Pi_2^2\cdot\1\{\Pi_1<k_n^{-\tfrac12},\Pi_2<k_n^{-\tfrac12}\}]\\
    \wle & k_n(k_n-1)\E[\Pi_1^2\Pi_2^2]\\
    \wsim & \frac{(\tau-1)^4}{k_n^2\tau^2(\tau-2)^2}.
    \end{align*}
    Putting everything together, we obtain the upper bound
    \[
    \E[(m_T^L)^2] \wle \E[m_T^L]^2 + o(n^4 k_n^{-2}).
    \]
    Furthermore, Jensen's inequality tells us that $\E[(m_T^L)^2]\ge \E[m_T^L]^2$, which implies
    \[
    \E[(m_T^L)^2]\wsim \E[m_T^L]^2,
    \]
    so that indeed $m_T^L/\E[m_T^L]\wprto1$.
    In conclusion, for $2<\tau\le3$,
    \[
    \frac{m_T}{\E[m_T]} \ \stackrel{w.h.p.}{=} \ 
    \frac{m_T^L}{\E[m_T]} \wsim \frac{m_T^L}{\E[m_T^L]} \wprto1.
    \]    
\end{proof}

%%%%%%%%%%%%%%%%%%%%%%%%%%%%%%%%%%%%%%%%%%%%%%%%%%%%%%%%%%%%%%%%%%%
%%                                                               %%
%% You may add acknowledgments (optional).                       %%
%%                                                               %%
%%%%%%%%%%%%%%%%%%%%%%%%%%%%%%%%%%%%%%%%%%%%%%%%%%%%%%%%%%%%%%%%%%%

%%%%%%%%%%%%%%%%%%%%%%%%%%%%%%%%%%%%%%%%%%%%%%%%%%%%%%%%%%%%%%%%%%%
%%                                                               %%
%% You have reached the end of your document.                    %%
%%                                                               %%
%%%%%%%%%%%%%%%%%%%%%%%%%%%%%%%%%%%%%%%%%%%%%%%%%%%%%%%%%%%%%%%%%%%

\end{document}